\documentclass[bj;preprint]{imsart}

\usepackage{amsmath,amsthm,epsfig,graphicx,color}
\usepackage{ifthen}
\usepackage{amssymb,latexsym}
\usepackage{subfigure}
\usepackage{hyperref}
\hypersetup{
urlcolor=blue, 
  menucolor=black, 
  citecolor=blue, 
  anchorcolor=black, 
  filecolor=black, 
  linkcolor=black, 
  colorlinks=true,
}
\usepackage{comment}
\usepackage[numbers,sort&compress]{natbib} 

\startlocaldefs
\newcommand{\sv}{stochastic volatility}
\newcommand{\bfS}{{\bf S}}
\newcommand{\stv}{\stackrel{v}{\rightarrow}}
\newcommand{\bbr}{{\mathbb R}}
\newcommand{\bfY}{{\bf Y}}
\newcommand{\kto}{k\to\infty}
\newcommand{\bth}{\begin{theorem}}
\newcommand{\ethe}{\end{theorem}}
\newcommand{\bre}{\begin{remark}\em }
\newcommand{\ere}{\end{remark}}
\newcommand{\rep}{representation}
\newcommand{\bfa}{{\mathbf a}}
\newcommand{\sign}{{\rm sign}}
\newcommand{\sn}{\sigma^{(n)}}
\newcommand{\ble}{\begin{lemma}}
\newcommand{\ele}{\end{lemma}}
\newcommand{\bbf}{{\mathcal F}}
\newcommand{\Ps}{{\mathbb P}_{\sigma}}
\def\Es{{\mathbb E}_{\sigma}}
\newcommand{\Pois}{{\rm Pois}}
\newcommand{\eqd}{\stackrel{d}{=}}

\newcommand{\wt }{\widetilde}

\newcommand{\ov}{\overline}

\newcommand{\pro}{probabilit}
\newcommand{\ld}{large deviation}

\newcommand{\bfv}{{\mathbf v}}
\newcommand{\bfA}{{\mathbf A}}
\newcommand{\bfT}{{\mathbf T}}

\newcommand{\bfR}{{\mathbf R}}
\newcommand{\bfD}{{\mathbf D}}

\newcommand{\bfN}{{\mathbf N}}

\newcommand{\bfM}{{\mathbf M}}
\newcommand{\bfF}{{\mathbf F}}

\newcommand{\bfQ}{{\mathbf Q}}
\newcommand{\bfI}{{\mathbf I}}
\newcommand{\bfX}{{\mathbf X}}

\newcommand{\fct}{function}
\newcommand{\slvary}{slowly varying}
\newcommand{\regvar}{regular variation}
\newcommand{\regvary}{regularly varying}
\newcommand{\st}{such that}

\newcommand{\std}{\stackrel{\rm d}{\rightarrow}}
\newcommand{\stp}{\stackrel{\P}{\rightarrow}}
\newcommand{\la}{\lambda}

\newcommand{\ds}{distribution}
\newcommand{\beao}{\begin{eqnarray*}}
\newcommand{\eeao}{\end{eqnarray*}}
\newcommand{\beam}{\begin{eqnarray}}
\newcommand{\eeam}{\end{eqnarray}}
\newcommand{\barr}{\begin{array}}
\newcommand{\earr}{\end{array}}

\definecolor{darkblue}{rgb}{.1, 0.1,.8}
\definecolor{darkgreen}{rgb}{0,0.8,0.2}
\definecolor{darkred}{rgb}{.8, .1,.1}
\newcommand{\bco}{\begin{corrolary}}
\newcommand{\eco}{\end{corrolary}}

\newcommand{\E}{\mathbb{E}}
\renewcommand{\P}{\mathbb{P}}
\newcommand{\1}{\mathbf{1}}
\newcommand{\R}{\mathbb{R}}
\newcommand{\N}{\mathbb{N}}

\newcommand{\bfC}{{\mathbf C}}

\DeclareMathOperator{\Proj}{Proj}

\newcommand{\X}{{\mathbf X}}
\newcommand{\Y}{{\mathbf Y}}

\newcommand{\A}{{\mathbf A}}
\newcommand{\bfZ}{{\mathbf Z}}

\newcommand{\twonorm}[1]{\|#1\|_2}

\newcommand{\frobnorm}[1]{\|#1\|_F}
\newcommand{\ltwonorm}[1]{\|#1\|_{\ell_2}}
\newcommand{\vep}{\varepsilon}
\newcommand{\nto}{n \to \infty}
\newcommand{\xto}{x \to \infty}

\newcommand{\rhs}{right-hand side}
\newcommand{\ts}{time series}
\newcommand{\tsa}{\ts\ analysis}
\newcommand{\fidi}{finite-dimensional distribution}
\newcommand{\rv}{random variable}

\newcommand{\diag}{\operatorname{diag}}

\newcommand{\bfe}{{\mathbf e}}

\newcommand{\ex}{{\rm e}\,}

\def\tag{\refstepcounter{equation}\leqno }

\newtheorem{lemma}{Lemma}[section]

\newtheorem{theorem}[lemma]{Theorem}

\newtheorem{remark}[lemma]{Remark}

\newcommand{\cip}{\stackrel{\P}{\rightarrow}}

\newcommand{\var}{{\rm var}}

\newcommand{\cov}{{\rm cov}}

\newcommand{\as}{{\rm a.s.}}

\newcommand{\evt}{extreme value theory}

\newcommand{\pp}{point process}
\newcommand{\con}{convergence}
\newcommand{\seq}{sequence}
\newcommand{\ms}{measure}
\newcommand{\asy}{asymptotic}
\endlocaldefs

\begin{document}

\begin{frontmatter}

\title{The eigenstructure of the sample covariance matrices of high-dimensional stochastic volatility models with heavy tails}
\runtitle{Sample covariance matrices of stochastic volatility models with heavy tails}


\author{\fnms{Johannes} \snm{Heiny}\ead[label=e1]{johannes.heiny@rub.de}}
\address{Department of Mathematics,\\
Ruhr-University Bochum,\\
Universit{\"a}tsstrasse 150,\\
D-44801 Bochum, Germany\\
\printead{e1}}
\and
\author{\fnms{Thomas} \snm{Mikosch}\ead[label=e2]{mikosch@math.ku.dk}}
\address{Department  of Mathematical Sciences,\\
University of Copenhagen,\\
Universitetsparken 5,\\
DK-2100 Copenhagen, Denmark\\
\printead{e2}}

\thankstext{}{Johannes Heiny thanks the Villum Foundation for 
partial financial support by Grant No 11745 ``Ambit fields: Probabilistic properties and statistical inference''. Thomas Mikosch's research is
generously supported by an Alexander von Humboldt Research Award.}
\affiliation{Ruhr-University Bochum and University of Copenhagen}

\runauthor{J.~Heiny and T.~Mikosch}

\begin{abstract}
We consider a $p$-dimensional time series 
where the dimension $p$ increases with the sample size $n$. The resulting 
data matrix $\X$ follows a stochastic volatility model: each entry consists of a positive
random volatility term multiplied by an independent noise term. The volatility multipliers introduce dependence in each row and across the rows. We study the asymptotic behavior of the eigenvalues and eigenvectors of the sample covariance 
matrix $\X\X'$ under a regular variation assumption on the noise. In particular, we prove Poisson convergence for the point process of the centered and normalized eigenvalues and derive limit theory for functionals acting on them, such as the trace. We prove related results for stochastic volatility models with additional linear 
dependence structure and for stochastic volatility models where the time-varying volatility terms are extinguished with high probability when $n$ increases. We provide explicit approximations of the eigenvectors which are of a strikingly simple structure. The main tools for proving these results are large deviation theorems for heavy-tailed time series, advocating a unified approach to the study of 
the eigenstructure of heavy-tailed random matrices.
\end{abstract}



\end{frontmatter}

\section{The stochastic volatility model}\setcounter{equation}{0}

Stochastic volatility models are popular in econometrics \cite{bauwens:hafner:laurent:2012}, mathematical finance \cite{andersen:davis:kreiss:mikosch:2009, fouque:papanicolaou:sircar:2011, fouque:sircar:2017} where they are used for option and derivative securities pricing, insurance mathematics \cite{cui:feng:mckay:2017, li:rong:zhao:2017}, time series \cite{mikosch:rezapour:2013, davis:mikosch:2001}, dependence modeling \cite{caviccioli:2017} and many other applied research areas. In a classical Black--Scholes framework the volatility is assumed constant. Empirical studies, however, have shown that many observed features of implied volatility surfaces, such as the so-called volatility smile, can only be explained by assuming a stochastic or even non-stationary volatility sequence over time; see for example the discussion in \cite{shephard:andersen:2009,shephard:2005}. Therefore a wide variety of stochastic volatility models has been proposed and well studied over the last few years. Stochastic volatility
models are heavily used within the fields of financial economics and mathematical finance to capture the impact of time-varying volatility on financial markets and decision making. Time-varying volatility is endemic in financial markets.
This was observed early on, for example by Mandelbrot \cite{mandelbrot:1963},
Fama \cite{fama:1965}, and Black and Scholes \cite{black:scholes:1972}.
\par 
The aforementioned literature on \sv\ 
models deals with univariate or low-dimensional multivariate 
time series. Here we focus on a high-dimensional \sv\ \ts\ whose dimension may grow  with the sample size.
To be precise, we study a $p$-dimensional \sv~ time series, and 
assuming that $p$ is large, we analyze the dependence structure of $n$ observations from this time series via spectral properties of the sample covariance matrix. We discuss two cases: a stochastic volatility field with dependence and whose marginal distribution does not change over time, and an iid stochastic volatility field with time-varying marginal distribution, both under the assumption of observations coming from a distribution with infinite fourth moment. This is quite a typical situation for financial and actuarial time series; see for example the Danish fire insurance data considered in \cite[Example 4.2]{resnick:2007}, emerging market stock returns \cite{ling:2005, hill:2013} and exchange rates data \cite{hill:2013b}.
For such time series it is also common to study so--called tail risk measures to describe the impact of extreme scenarios \cite{kelly:jiang:2014, bollerslev:todorov:xu:2015}.

In the first part of this paper, we consider the $p\times n$-dimensional  data matrix 
\beao
\bfX=\bfX_n=\big(X_{it}\big)_{i=1,\ldots,p;t=1,\ldots,n}\,,
\eeao
where $(X_{it})$ has the structure of a {\em \sv\ model}, i.e.,
\beam\label{eq:svmodel}
X_{it}=\sigma_{it}\,Z_{it}\,,\qquad i,t=1,2,\ldots\,,
\eeam
and $(\sigma_{it})$ is a strictly stationary random field of non-negative \rv s  
independent of the iid random field $(Z_{it})$. In Section~\ref{sec:more}, we introduce additional dependence in the \sv\ model. In what follows, $X,\sigma,Z,$ denote generic elements of these fields.
Stochastic volatility models are common in financial \tsa ; see for example \cite{andersen:davis:kreiss:mikosch:2009}. The present model
is an extension allowing for dependence through time and across the rows of the data matrix. It is convenient to 
think of \eqref{eq:svmodel} as a model where each row stands for a \ts\ of log-returns of
a speculative price series from a large portfolio, e.g. a stock index such as the 
Standard \& Poors 500 where each of the 500 rows of $\X$ could represent 
the log-returns
of the stock price of a particular US-based company in a given period of time.
\par
We will
study the {\em eigenstructure}, that is eigenvalues and eigenvectors, of the $p\times p$ {\em sample covariance matrix} $\bfS=\bfX\bfX'$ with entries
\beao
S_{ij}= \sum_{t=1}^n X_{it}X_{jt}\,, \quad i,j=1,\ldots,p\,,
\eeao
under the assumption that the dimension $p=p_n$ converges to infinity together with the sample size $n$. 
In what follows, we drop the double index for the diagonal entries $S_{ii}$ and simply write $S_i$.
In the model \eqref{eq:svmodel} the dependence
across the rows and through time is described by the  structure of the 
volatility field $(\sigma_{it})$. 
We will assume that the noise variable
$Z$ is {\em heavy-tailed} in the sense that it satisfies the \regvar\ condition
\beam\label{eq:regvar}
\P(Z>x)\sim q_+ \dfrac{L(x)}{x^{\alpha}}\quad\mbox{and}\quad  \P(Z< -x)\sim q_-
\dfrac{L(x)}{x^{\alpha}}\,,\qquad \xto\,,
\eeam
for some $\alpha\in(0,4)$, constants $q_+,q_-\ge  0$ \st\  $q_++q_-=1$, and a \slvary\ \fct\ $L$. 
We  assume $\E[ Z]=0$ whenever $\E[|Z|]<\infty$
and also that the non-negative $\sigma$ has a much lighter tail than $Z$ in the sense that all moments of $\sigma$ are finite.
\par
The considered random field $(X_{it})$ is flexible as regards second order dependence. 
If $\alpha>1$, 
we have $\cov(X_{it},X_{js})=0$ 
for $(i,t)\ne (j,s)$.  On the other hand,
$\cov(|X_{it}|^r,|X_{js}|^r)$, $r>0$, may decay arbitrarily slowly to zero when $|i-j|$ or $|t-s|$ goes to infinity, provided these covariances exist. Arbitrary decay rates can be achieved, 
for example, by assuming that $(\log \sigma_{it})$ is a stationary Gaussian field with a suitable
covariance structure. As a matter of fact, a large part of the literature on \sv\ \ts\ models
deals with the case when the log-volatility is stationary Gaussian; 
see \cite{andersen:davis:kreiss:mikosch:2009} for surveys on the topic \sv . 
\par
Thanks to \regvar\ and the iid-ness of the noise $(Z_{it})$,
the extremal dependence structure of $(X_{it})$ is characterized by the fact that the \fidi s of
$(X_{it})$ are multivariate \regvary\ with index $\alpha$ and have  \asy ally independent marginals; 
we refer to \cite{resnick:1987,resnick:2007,buraczewski:damek:mikosch:2016} for introductions
to multivariate \regvar . 
Indeed, applications of Breiman's lemma (Lemma B.5.1 in \cite{buraczewski:damek:mikosch:2016}) imply that
\beam\label{eq:breim}
\P( \pm \sigma_{it}Z_{it}>x)\sim \E [\sigma^\alpha]\,\P(\pm Z>x)\,,\qquad \xto\,.
\eeam
Thus the marginal \ds s are \regvary\ with index $\alpha$. Moreover,
for $(i,t)\ne (j,s)$, by another Breiman argument,
\beao
\dfrac{\P(\sigma_{it}|Z_{it}| >x\,,\sigma_{js}|Z_{js}|>x)}{\P(|Z|>x)}&=&
\dfrac{\P(\min(\sigma_{it}|Z_{it}|,\sigma_{js}|Z_{js}|)>x)}{\P(|Z|>x)}\\
&\le &\dfrac{\P(\max(\sigma_{it},\sigma_{js})\,\min(|Z_{it}|,|Z_{js}|)>x)}{\P(|Z|>x)}\\
&\sim & \E\big[\max(\sigma_{it},\sigma_{js})^\alpha\big]\,\P(|Z|>x) \to 0\,,\qquad \xto\,.
\eeao
This means that $X_{it}$ and $X_{js}$ are \asy ally independent in the sense of extreme value theory. 
Writing
\beao
\bfX^{(d)}= (X_{it})_{i,t=1,\ldots,d}\,,\qquad \bfZ^{(d)}= (Z_{it})_{i,t=1,\ldots,d}\,,
\qquad d\ge 1\,,
\eeao
the previous calculations on the marginals combined with standard arguments 
from \regvar\ calculus (see \cite{resnick:1987,resnick:2007,buraczewski:damek:mikosch:2016}) ensure that
\beao
\dfrac{\P(x^{-1}\bfZ^{(d)} \in\cdot)}{\P(|Z|>x)}\stv \nu_\alpha(\cdot)\,,\quad
\dfrac{\P(x^{-1}\bfX^{(d)} \in\cdot)}{\P(|Z|>x)}\stv \E [\sigma^\alpha]\,\nu_\alpha(\cdot)\,,\qquad \xto\,.
\eeao
Here $\stv$ denotes vague \con\ in $\ov \bbr^{d\times d}\backslash\{\bf0\}$, $\ov \bbr=\bbr\cup\{\infty,-\infty\}$,  the limiting \ms\ $\nu_\alpha$ is concentrated on the axes,  and its
restriction to any of the axes has Lebesgue density given by
\beao
\alpha \,|x|^{-\alpha-1}\,\big[q_+\,\1 (x>0)+q_-\,\1 (x<0)\big]\,.
\eeao
The fact that $\nu_\alpha$ is concentrated on the axes is another way of defining 
\asy\ independence of the components of $\bfX$.
\par
Since we are interested in the sample covariance matrix $\bfS$ in the heavy-tailed case we observe that
its diagonal entries $S_i=\sum_{t=1}^n X_{it}^2$ and off-diagonal entries $S_{ij}=\sum_{t=1}^nX_{it}X_{jt}$ for $i\ne j$ have rather distinct tails.
A first indication is the fact that, on one hand, by a Breiman argument,
\beam\label{eq:tail1}
P(X^2>x) \sim  \E[\sigma^\alpha] \,\P(Z^2>x)\sim \E[\sigma^\alpha] \,x^{-\alpha/2}\, L(\sqrt{x})\,,
\eeam
while, on the other hand, by a result in Embrechts and Goldie \cite{embrechts:goldie:1980}, 
for independent copies $Z_1,Z_2$ of $Z$,
\beam\label{eq:z1z2}
\P( Z_1Z_2>x)\sim  \wt q_+ \dfrac{\ell(x)}{x^\alpha}\quad\mbox{and}\quad\P( Z_1Z_2<-x)\sim  \wt q_- \dfrac{\ell(x)}{x^\alpha}\,,
\eeam
where $\ell$ is a \slvary\ \fct , $\wt q_+,\wt q_-\ge 0$ and $\wt q_++\wt q_-=1$.
Hence by Breiman's lemma, 
 for $(i,t)\ne (j,s)$,
\beam\label{eq:tail2}
\P(\pm X_{it}X_{js}>x)\sim \E \big[(\sigma_{it}\sigma_{js})^{\alpha}\big]\,\P(\pm Z_1Z_2>x)\,,\qquad \xto\,.
\eeam 
We assume $\alpha\in (0,4)$. In this case, \eqref{eq:tail1} and \eqref{eq:tail2} imply  
that the diagonal entries $(S_i)$ of $\bfS$ dominate all off-diagonal elements $S_{ij}$ 
in the sense that
the \asy\ behavior of the eigenvalues of  $\bfS$ is completely determined by the diagonal  
$\diag(\bfS)$ of $\bfS$. This phenomenon is described in Theorem~\ref{thm:diag}. 
It is well known in the iid case when $p=p_n\to\infty$ 
(see \cite{davis:mikosch:pfaffel:2016,davis:heiny:mikosch:xie:2016,heiny:mikosch:2017:iid}). Pioneering work for the largest eigenvalue of $\bfS$ under a more restrictive growth condition on $p$ and $\alpha\in (0,2)$
is due to Soshnikov~\cite{soshnikov:2004,soshnikov:2006}
and Auffinger et al.~\cite{auffinger:arous:peche:2009}.
For constant $p$ the same 
property was observed for the \sv\ model \eqref{eq:svmodel} in Jan\ss en et al. \cite{janssen:mikosch:rezapour:xie:2017}.
\par
The diagonal elements $S_i$ are the eigenvalues of the matrix 
$\diag(\bfS)$. They approximate the eigenvalues of the sample covariance matrix $\bfS$; see
\eqref{eq:weyl}. Given this approximation, 
\ld\ results from Mikosch and Wintenberger \cite{mikosch:wintenberger:2016,mikosch:wintenberger:2013}
for the partial sums $S_i$ are used to derive the \con\ of the point processes 
of the centered and normalized eigenvalues of $\bfS$ towards an inhomogeneous Poisson
process; see Theorem~\ref{thm:pp}.
A similar point process convergence in the iid case under the assumption that $p$ and $n$ are proportional was proved in \cite{soshnikov:2004,soshnikov:2006} for $\alpha\in (0,2)$ and later extended in \cite{auffinger:arous:peche:2009} to $\alpha \in [2,4)$. In their proofs the authors used truncation techniques and a challenging combinatorial approach.

Based on Theorem~\ref{thm:pp}, the \con\ of the \pp\ of the eigenvalues 
in the case $\alpha\in (0,4)$ allows one to 
derive limit theory for the largest eigenvalues of $\bfS$ and \fct als acting on them. 
In particular, the centered and normalized largest eigenvalue of $\bfS$ 
converges to a Fr\'echet distributed  \rv\ with parameter $\alpha/2$. 
In \cite{heiny:mikosch:2017:iid}, this was shown for an iid random field $(X_{it})$. 
\par
In Section~\ref{sec:more}, we introduce additional dependence in the \sv\ model. We consider the $p\times p$ matrix
$\Y=\A^{1/2}\X$ where $\bfA=\bfA_n$ are deterministic positive definite $p\times p$ matrices with
uniformly bounded spectra. In Theorem~\ref{thm:mainsigma2} it is essentially shown that the
eigenvalues of $\bfY\bfY'$ are approximated by those of the matrix $\diag(\bfS)\diag(\bfA)$
and Theorem~\ref{thm:eigenvector} yields an approximation for the 
eigenvectors of $\bfY\bfY'$.
\par
In Section~\ref{sec:thin}, we consider another modification of the \sv\ model \eqref{eq:svmodel}.
We assume that the distribution of $\sigma$ is a \fct\ of $n$ and write $\sigma^{(n)}$ for a generic random variable from the iid random field $(\sigma_{it}^{(n)})$, $n\ge 1$.
The possible values of $\sigma^{(n)}$, $n\ge 1$, are $0=s_0<s_1\cdots <s_m$ for some $m\ge 1$ and we assume
that $q_0^{(n)}=\P(\sigma^{(n)}=0)\to 1$ and that the limits $\lim_{\nto}n\,\P(\sigma^{(n)}=s_j)>0$, $j=1,\ldots,m$, 
exist (finite or infinite). This means that there is a large  \pro y of extinction 
of the iid entries  $X_{it}^{(n)}=\sigma_{it}^{(n)}Z_{it} $ 
of the data matrix $\bfX_n$ when $n$ is large. This model was 
introduced in \cite{auffinger:tang:2016} for $m=1$, $1-q_0^{(n)}=n^{-v}$ 
for some $v\in (0,1]$ and $p/n\to \gamma\in (0,\infty)$. In Theorem~\ref{thm:diagsigma}, we again show that the eigenvalues of
$\bfS$ are \asy ally given by $\diag(\bfS)$. The main difference to Theorem~\ref{thm:diag} is that the
normalization needed for the eigenvalues of $\bfS$ is of  significantly 
smaller magnitude depending on the speed at which  $q_0^{(n)}$ 
approaches 1.  
The method of proof of our results is different from those
in \cite{auffinger:tang:2016} and works for more general growth rates of $p$; we again use \ld\ techniques and exploit the approximation
of the eigenvalues of $\bfS$ by those of $\diag(\bfS)$. We also derive the \pp\ \con\
of the eigenvalues of $\bfS$, find approximations for the eigenvectors and we derive results for $\bfY=\bfA^{1/2} \bfX$ where $\bfA$ is a deterministic positive
definite matrix.
\par
In Sections~\ref{sec:proofthm21}--\ref{sec:proofthm3.2}, we provide the proofs 
of the 
aforementioned results. 

\subsection*{Some basic notation}
\subsubsection*{Eigenvalues and eigenvectors} For any $p\times p$ positive semidefinite matrix $\bfC$, we denote its ordered eigenvalues by 
\begin{equation*}
\la_1(\bfC) \ge \cdots \ge \la_p(\bfC)\,.
\end{equation*}
If, for $k\le p$, the multiplicity of $\la_k(\bfC)$ is 1, then there exists a unique unit eigenvector $\bfv_k(\bfC)$ associated with $\la_k(\bfC)$, i.e. $\ltwonorm{\bfv_k(\bfC)}=1$ (Euclidean norm) and 
\begin{equation*}
\bfC \bfv_k(\bfC)= \la_k(\bfC) \bfv_k(\bfC)\,,
\end{equation*}
such that the first non-zero coordinate of $\bfv_k(\bfC)$ is positive. We will use the latter orientation convention throughout this paper for eigenvectors.

\subsubsection*{Spectral norm and diagonal matrix}
For any $p\times p$ matrix $\bfC$, the spectral norm $\twonorm{\bfC}$ is $\sqrt{\la_1(\bfC \bfC')}$. Moreover, $\diag(\bfC)$ denotes the diagonal matrix which has the same diagonal as $\bfC$. Sometimes we will simply refer to $\diag(\bfC)$ as the diagonal of $\bfC$.
\subsubsection*{Normalization} Typically, we use a \seq\ $(a_k)$ satisfying 
$k\,\P(|Z|>a_k)\to 1$ as  $\kto$ for the normalization of eigenvalues.

\section{Convergence results for the \sv\ model}\label{sec:mainresults}\setcounter{equation}{0}
We start with a fundamental approximation of the sample covariance matrix $\bfS$ in spectral norm.
\bth\label{thm:diag} Consider the \sv\ model \eqref{eq:svmodel}. We assume the following conditions:
\begin{enumerate}
\item A growth condition for  the integer sequence $p=p_n\to\infty$:
\begin{equation}\label{eq:p}
p=p_n=n^\beta \ell(n), \quad n\ge1,\tag{$C_p(\beta)$}
\end{equation}
where $\ell$ is a slowly varying function and $\beta\in (0,1]$. 
\item
The \regvar\ condition \eqref{eq:regvar} on $Z$ for some $\alpha\in (0,2)\cup (2,4)$ and 
$\E[Z]=0$ if $\E[|Z|]<\infty$.
\item Finiteness of all moments $\E [\sigma^{r}]$ for $r>0$.
\end{enumerate}
Then
\beao
a_{np}^{-2}\twonorm{\bfS-\diag(\bfS)} \stp 0\,,\qquad\nto\,.
\eeao
\ethe
This theorem provides a first indication that the spectral properties of $\bfS$ might be similar to those of $\diag(\bfS)$ which has a simple structure. 
The normalizing sequence is of the form $a_{np}^2=(np)^{2/\alpha} \ell_1(np)$ for some slowly varying function $\ell_1$. 
Note that the provided approximation of $\bfS$ does not hold for $\alpha>4$ when the fourth moment of $X$ is finite. In fact, one obtains completely different types of limit results for the eigenstructure of $\bfS$; see \cite{davis:heiny:mikosch:xie:2016,heiny:mikosch:2017:iid} and the monograph \cite{bai:silverstein:2010} for a detailed overview and more references. The approximation of the sample covariance matrix by its diagonal is featured in the heavy-tailed case only.

The proof of Theorem~\ref{thm:diag} is provided in Section~\ref{sec:proofthm21}.
\bre\label{rem:other}
Assume $\beta>1$ in \eqref{eq:p}. If we keep the remaining assumptions of Theorem~\ref{thm:diag},
the same proof as for the latter result yields
 \beao
a_{np}^{-2}\,\|\bfX'\bfX-\diag(\bfX'\bfX)\|_2 \stp 0\,,\qquad\nto\,.
\eeao
On the other hand, the non-zero eigenvalues of $\bfS=\bfX\bfX'$ and   $\bfX'\bfX$ are the same. This observation
is useful when determining the \asy\ behavior of the eigenvalues of $\bfS$ in the case $\beta>1$.
\ere
In view of Weyl's inequality (see \cite{bhatia:1997}), we may conclude from 
Theorem~\ref{thm:diag} that
\beam\label{eq:weyl}
a_{np}^{-2}\max_{i=1,\ldots,p} \big|\la_{i}(\bfS)-\la_i(\diag(\bfS))\big|\le 
a_{np}^{-2}\twonorm{\bfS-\diag(\bfS)} \stp 0\,,\qquad\nto\,.
\eeam
Using \eqref{eq:weyl}, it is possible to study the \asy\ behavior of the 
\pp\ of the scaled eigenvalues $(a_{np}^{-2}\la_i(\bfS))_{i=1,\ldots,p}$, as formulated in the next theorem.
\bth\label{thm:pp}
Assume the conditions of Theorem~\ref{thm:diag}. In addition, we assume the following conditions.
\begin{enumerate}
\item[\rm (4)] 
$(\sigma_{it})$ is a strictly stationary ergodic field and the sequence $(\sigma_{1t}^2)$ is strongly mixing with rate \fct\ 
$\alpha_j\le c j^{-a}$ for some 
constants $c>0$ and $a>1$.
\item[\rm (5)] 
$\sigma^2\le M$ a.s.
for some constant $M>0$. 
\end{enumerate}
Then we have the following weak \con\ result for the \pp es
with state space $\bbr\backslash \{0\}$:
\beao
N_n= \sum_{i=1}^p \vep_{a_{np}^{-2} (\la_i(\bfS)-c_n)}\std N
\,,\qquad\nto\,,
\eeao 
where $N$ is a Poisson process on $\bbr\backslash \{0\}$ with mean \ms\ $\mu_\alpha(x,\infty)=\E[\sigma^\alpha]\,x^{-\alpha/2}$
and $\mu_\alpha(-\infty,-x)=0$ for  $x>0$. Furthermore, $\vep_x$ denotes the Dirac measure in the point $x$ and  
\beam\label{eq:cn}
c_n=\left\{ \barr{ll}
0\,,& \mbox{if $\alpha\in (0,2)$,}\\
n\,\E [X^2]\,,&\mbox{if $\alpha\in (2,4)$}\,. 
\earr\right.\,
\eeam
\ethe
The proof will be given in Section~\ref{sec:proofthm24}.
We notice that this result is the same as for the iid field $((\E[\sigma^\alpha])^{1/\alpha}\,Z_{it})$; see \cite[Theorem 3.10 and Lemma 3.8]{heiny:mikosch:2017:iid}. This means that dependence within the light-tailed $\sigma$-field influences the limiting point process only through a multiplicative factor. 
\bre In view of Remark~\ref{rem:other}, an analogous result holds if $\beta>1$ in \eqref{eq:p}.
\ere
The limiting process in Theorem \ref{thm:pp} has \rep
\beam\label{eq:N1}
N=\sum_{i=1}^\infty \vep_{(\Gamma_i/\E[\sigma^\alpha])^{-2/\alpha}}\,,
\eeam
where $\Gamma_i=E_1+\cdots + E_i$ for iid standard exponential \rv s $(E_i)$. From this result it follows that
\beam\label{eq:N2}
a_{np}^{-2} (\E[\sigma^\alpha])^{2/\alpha}\,\big(\la_1(\bfS) -c_n,\ldots,\la_{k}(\bfS)-c_n\big)\std \big(\Gamma_1^{-2/\alpha}\,,\ldots,\Gamma_k^{-2/\alpha}\big) 
\eeam
for fixed $k\ge 1$. In particular,
\beao
a_{np}^{-2} (\E[\sigma^\alpha])^{2/\alpha}\,\big(\la_{1}(\bfS) -c_n\big)\std \Gamma_1^{-2/\alpha}\,,
\eeao
and the limiting variable has a Fr\'echet distribution with parameter $\alpha/2$.
Now one can apply the folklore from \evt\ to derive limit theory for continuous  \fct als of $(\la_{1}(\bfS),\ldots,\la_{k}(\bfS))$. Moreover, a continuous mapping argument also shows that
\beao
a_{np}^{-2}\sum_{i=1}^p (\la_{i}(\bfS)-c_n) = a_{np}^{-2}\big({\rm trace} \big(\bfS\big)-p c_n\big) 
\eeao   
converges in \ds\ to a totally skewed to the right $\alpha/2$-stable limit; see \cite{davis:heiny:mikosch:xie:2016, heiny:mikosch:2017:iid,davis:mikosch:pfaffel:2016}.


\section{Introducing more dependence in the stochastic volatility model}\setcounter{equation}{0}\label{sec:more}
In this section, we will extend our stochastic volatility model by including some additional dependence between the  entries of $\bfX$. 

To this end, let $\A=\A_n$ be a sequence of deterministic, positive definite $p\times p$ matrices with bounded spectrum, that is $(\twonorm{\A_n})$ is uniformly bounded. 
If the entries of $\X$ are independent with mean $0$ and variance $1$, then the columns of 
\begin{equation}\label{eq:defY}
\Y=\A^{1/2}\X\,
\end{equation} 
have covariance matrix $\A$. Here $\A^{1/2}$ is the symmetric, positive definite square root of~$\A$. 
\begin{remark}{\em
The positive definite $\bfA$ can be diagonalized: $\bfA=\mathbf{O} \bfT \mathbf{O}'$ where $\mathbf{O}$ is an orthogonal matrix and $\bfT$ is diagonal and positive definite. By assumption, $\bfT^{1/2}$ exists and we get $\bfA^{1/2}= \mathbf{O} \bfT^{1/2} \mathbf{O}'$.
}\end{remark}

The transformation \eqref{eq:defY} is very important in multivariate statistics since it creates a sample with  dependence structure $\bfA$ from an iid sample and vice versa.

Now assume that $\X$ follows the stochastic volatility model \eqref{eq:svmodel}. While the dependence among the $(X_{it})$ is only due to the dependence among the light-tailed $(\sigma_{it})$, the dependence of the heavy-tailed components in the entries of $\Y=(Y_{it})$ is determined by $\A$.   
Our main goal in this section is to approximate the eigenvalues and eigenvectors of 
\beao
\Y\Y' =\A^{1/2}\X \X' \A^{1/2}=\A^{1/2}\bfS \A^{1/2}\,.
\eeao

As regards eigenvalues, we note that the spectra of $\A^{1/2}\bfS \A^{1/2}$ and $\bfS \A$ coincide.
Matrices, such as $\bfS \A$, which are a product of a sample covariance matrix and the inverse of another covariance matrix are called {\em multivariate $F$-matrices} \cite{bai:silverstein:2010}. The limiting spectral distribution of $F$-matrices was studied among others in \cite{yin:bai:krishnaiah:1983}. $F$-matrices also play an important role in MANOVA. Wachter \cite{wachter:1980} analyzed the generalized eigenvalue problem
\begin{equation}\label{eq:genev}
\det (\bfS -\la \A^{-1})=0\,,
\end{equation}
where $\A$ can be stochastic but is independent of $\X$. Since $\A$ is positive definite its inverse can be interpreted as a covariance matrix. Solutions $\la$ of \eqref{eq:genev} are eigenvalues of $\A^{1/2}\bfS \A^{1/2}$, see \cite{debashis:aue:2014,bai:silverstein:2010}.

The entries of the matrix $\Y$ possess a quite general dependence structure. Nevertheless the approximation of the eigenvalues of the associated sample covariance matrix $\Y\Y'$ is straightforward.  
\begin{theorem}\label{thm:mainsigma}
We consider the matrix $\Y=\A^{1/2} \X$, where $\X$ follows the stochastic volatility model \eqref{eq:svmodel}. We assume the following conditions:
\begin{itemize}
\item The growth condition \eqref{eq:p} with $\beta\in (0,1]$. 
\item
The \regvar\ condition \eqref{eq:regvar} on $Z$ for some $\alpha\in (0,2)\cup (2,4)$ and 
$\E[Z]=0$ if $\E[|Z|]<\infty$.
\item Finiteness of all moments $\E [\sigma^{r}]$ for $r>0$.
\item $\bfA=\A_n$ constitutes a sequence of deterministic, positive definite $p\times p$ matrices with uniformly bounded spectra.
\end{itemize}
 Then
\beao
a_{np}^{-2}\,\max_{i=1,\ldots,p}\big|\la_{i}(\A^{1/2}\bfS \A^{1/2})-\la_{i}(\diag(\bfS)\A)\big|\stp 0\,.
\eeao
\end{theorem}
\begin{proof}
By Weyl's inequality (see \cite{bhatia:1997}), Theorem~\ref{thm:diag} and the uniform boundedness of $(\twonorm{\A})$, we have
\begin{equation*}
\begin{split}
a_{np}^{-2}\,\max_{i=1,\ldots,p}&\big|\la_{i}(\A^{1/2}\bfS \A^{1/2})-\la_{i}(\diag(\bfS)\A)\big| \le a_{np}^{-2} \twonorm{\bfS \A - \diag(\bfS)\A}\\
&\le a_{np}^{-2} \twonorm{\bfS - \diag(\bfS)} \twonorm{\A} \cip 0\,,\qquad\nto\,.
\end{split}
\end{equation*}
\end{proof}

In applications involving high-dimensional data sets, it is common to only allow for dependence between certain key variables, which corresponds to many entries of $\bfA$ being zero. Therefore, we introduce a sparseness condition on $\A$ under which we can derive asymptotic spectral properties of $\diag(\bfS)\A$. 

We say that $\bfA=(A_{ij})\in \R^{p\times p}$ is a {\em band matrix} with {\em bandwidth} $m$ if $A_{ij}=0$ whenever $|i-j|>m$.
If $A_{1\bullet},\ldots,A_{p\bullet}\in \R^{1\times p}$ denote the rows of $\A$, we have 
\begin{equation*}
\diag(\bfS)\A= (S_1 A_{1\bullet}', \ldots, S_p A_{p\bullet}')'\,.
\end{equation*}
For $1\le k\le p$, there are $\binom{p}{k}$ ways to choose $k$ of the $p$ rows of $\bfA$. Each choice is uniquely described by an element of the set
\begin{equation*}
\Pi_{k,p}=\{\bfa=(a_1,\ldots,a_k)\in \{1,\ldots,p\}^k: a_1<\cdots<a_k\}\,,
\end{equation*}
where the coordinates of $\bfa$ contain the indices of the selected $\A_{i\bullet}$.
 For $\bfa \in \Pi_{k,p}$ define
\begin{eqnarray*}
J_{k,p}(\bfa,\bfA)= \left\{\begin{array}{ll}
1\,,& \text{if } \sum_{i=1}^k \sum_{j=1; |j- a_i|>k}^p |A_{a_i,j}|>0\,,\\
0\,,&\mbox{otherwise}\,.
\end{array}\right.
\end{eqnarray*}
\begin{remark}{\em
In other words, $J_{k,p}(\bfa,\bfA)$ is 0 if, after inspection of the rows $A_{a_1\bullet},\ldots, A_{a_k\bullet}$ and no further information about $\A$, it is still possible that $\A$ is a band matrix with bandwidth $k$. In fact, $\A$ is a band matrix with bandwidth $k$ if and only if $J_{k,p}(\bfa,\bfA)=0$ for all
$\bfa\in \Pi_{k,p}$.
Also note that $A_{ii}>0$ for all $i$ since $\A$ is symmetric and positive definite. 
}\end{remark} 
For $\wt \bfa  \in \Pi_{k,p}$ chosen uniformly at random, the probability $\P(J_{k,p}(\wt \bfa,\bfA)=1)$ is given by
\begin{equation*}
P_n(\bfA,k):= \binom{p}{k}^{-1} \sum_{\bfa\in \Pi_{k,p}} J_{k,p}(\bfa,\bfA)\,.
\end{equation*}

The following condition holds if the matrices $(\A_n)$ are ``nearly banded''.

{\bf Condition} \eqref{SB}: 
For the sequence of matrices $(\bfA_n)$ 
\begin{equation}\label{SB}
\text{there exists a sequence } k=k_p\to\infty\,,k_p^3=o(p)  \text{ such that } \lim_{\nto} P_n(\bfA,k) =0\,.\tag{$NB$}
\end{equation}
By construction, a sequence of band matrices $(\A)$ with bandwidths $(k)$ such that $k_p^3=o(p)$ satisfies condition \eqref{SB} since $P_n(\bfA,k) =0$ for all $n$. Roughly speaking, $P_n(\bfA,k)$ is small if only a small number of rows relative to the dimension $p$ violate the band matrix structure. In particular, a change of finitely many rows does not influence the validity of condition \eqref{SB}. 

Under condition \eqref{SB} we can simplify $\la_{i}(\diag(\bfS)\A)$ which appeared as approximation of the eigenvalues of $\bfY \bfY'$ in Theorem~\ref{thm:mainsigma}.
We have the following result.
\begin{theorem}[Eigenvalues of $\Y\Y'$]\label{thm:mainsigma2}
Consider the setting and the conditions of Theorem~\ref{thm:mainsigma}. In addition, we assume the following:
\begin{itemize}
 \item $(\A_n)$ satisfies condition \eqref{SB}.
\item The rows of
$(\sigma_{it})_{i,t\ge 1}$ are iid, strictly stationary ergodic sequences. Moreover, they are strongly mixing with rate \fct\ 
$\alpha_j\le c j^{-a}$ for some 
constants $c>0$ and $a>1$. 
\end{itemize}
\begin{enumerate}
\item If $\alpha\in (0,2)$, then
\begin{equation}\label{eq:ev}
a_{np}^{-2}\,\max_{i=1,\ldots,p}\big|\la_{i}(\A^{1/2}\bfS \A^{1/2})-\la_i(\diag(\bfS)\diag(\A))\big|\stp 0\,.
\end{equation}

\item If $\alpha\in (2,4)$, then
\begin{equation*}
a_{np}^{-2}\,\max_{i=1,\ldots,p}\big|\la_{i}(\A^{1/2}(\bfS-c_n \bfI) \A^{1/2})-\la_i(\diag(\bfS-c_n \bfI)\diag(\A))\big|\stp 0\,,
\end{equation*}
with centering $c_n$ defined in \eqref{eq:cn}.
\end{enumerate}
\end{theorem}
While $\bfY\bfY'=\A^{1/2}\bfS \A^{1/2}$ is a product of large matrices with complicated eigenstructure, the eigenvalues of $\diag(\bfS)\diag(\A)$ are very easy to find. 
\begin{remark}{\em
In the case $\alpha \in (2,4)$ we note that $\A^{1/2}(\bfS-c_n \bfI) \A^{1/2} = \bfY\bfY' -\E[\bfY\bfY']$. If $\A=\diag(\A)$ then the centering is not needed and \eqref{eq:ev} also holds for $\alpha \in (2,4)$; compare with \cite[Theorem 3.11]{heiny:mikosch:2017:iid}.
}\end{remark}
\begin{proof}
We start with the case $\alpha\in(0,2)$. Let $(k)$ be the integer sequence from condition \eqref{SB}. Since $k\to \infty$ we have $a_{np}^{-2} \la_{k}(\diag(\bfS)) \cip 0$ which implies  
\begin{equation*}
a_{np}^{-2}\la_{k+1}(\diag(\bfS)\A) \le a_{np}^{-2}\la_{k+1}(\diag(\bfS)) \twonorm{\A} \cip 0\,.
\end{equation*}
Therefore it is sufficient to prove
\begin{equation}\label{eq:sldfk1}
a_{np}^{-2} \max_{i=1,\ldots,k} |\la_{i}({\diag_k(\bfS)}\A)-\la_i(\diag_k(\bfS)\diag(\A))| \cip 0\,,
\end{equation}
where ${\diag_k(\bfS)}$ is created from ${\diag(\bfS)}$ by only keeping its $k$ largest entries and setting the others to 0.

Define the random indices $L_1,\ldots, L_p$ via
\begin{equation}\label{eq:L}
S_{L_1}=\la_1(\diag(\bfS)) > \cdots >S_{L_p}=\la_p(\diag(\bfS)) \quad \as\,
\end{equation}
In other words, $S_{L_i}$ is the $i$th order statistic of $S_1,\ldots,S_p$.
We have
\begin{equation*}
{\diag_k(\bfS)}\A= (0_p, \ldots, 0_p, S_{\pi_1} A_{{\pi_1}\bullet}',0_p, \ldots, 0_p, S_{\pi_2} A_{{\pi_2}\bullet}', \ldots, S_{\pi_k} A_{{\pi_k}\bullet}',0_p, \ldots, 0_p)'\,,
\end{equation*}
where $\pi_1 < \ldots < \pi_k$ are the order statistics of $L_1,\ldots,L_k$ and $0_p$ is the $p$-dimensional zero vector. Since the $S_i$'s are iid, $L_1,\ldots,L_k$ have a uniform distribution
on the set of
distinct $k$-tuples from $(1,\ldots,p)$. Therefore the $k$-tuple $\pi=(\pi_1,\ldots,\pi_k)$ is uniformly distributed on $\Pi_{k,p}$. 

Define the set 
$B_n=\{ J_{k,p}(\pi,\bfA)=0\}$. 
From condition \eqref{SB} and the fact that $\pi$ is uniformly distributed on $\Pi_{k,p}$, we see that $\P(B_n)\to 1$.
On $B_n$, we have for $1\le i\le k$,
\begin{equation*}
S_{\pi_i} A_{{\pi_i}\bullet}= (0,\ldots,0, S_{\pi_i} A_{{\pi_i},\pi_i-k},
S_{\pi_i} A_{{\pi_i},\pi_i-k+1}, \ldots, S_{\pi_i} A_{{\pi_i},\pi_i+k},0,\ldots,0)\,.
\end{equation*}

 Consider the set
\begin{equation}\label{eq:setA}
C_n=\{|L_i-L_j|>2k\,, i,j=1,\ldots,k\,, i\ne j\}\,.
\end{equation}
Since $L_1,\ldots,L_k$ are uniformly distributed
on the set of
distinct $k$-tuples from $(1,\ldots,p)$ we have
\begin{equation*}
\lim_{\nto} \P(C_n^c)\le \lim_{\nto} k(k-1) \dfrac{2pk (p-2)\ldots (p-k+1)}{p(p-1)\ldots (p-k+1)}\le \lim_{\nto}
\dfrac{2k^3}{p-1}= 0 \,,
\end{equation*}
where condition \eqref{SB} was used for the last equality.

On $B_n\cap C_n$, the matrix ${\diag_k(\bfS)}\A$ is block diagonal with $(2k+1)\times (2k+1)$ blocks $\bfQ_i$, $i\le k$. The matrix $\bfQ_i$ is zero everywhere except for its $(k+1)$st row which is  
\begin{equation*}
(S_{\pi_i} A_{{\pi_i},\pi_i-k},
S_{\pi_i} A_{{\pi_i},\pi_i-k+1}, \ldots, S_{\pi_i} A_{{\pi_i},\pi_i+k})\,,\quad i\le k\,.
\end{equation*}
The $(k+1, k+1)$ entry of $\bfQ_i$ is at position $(\pi_i,\pi_i)$ of ${\diag_k(\bfS)}\A$.
Therefore the only non-zero eigenvalue of $\bfQ_i$ is $S_{\pi_i} A_{{\pi_i},\pi_i}$. 
We conclude that on $B_n\cap C_n$
\begin{equation}\la_{i}({\diag_k(\bfS)}\A)=\la_i(\diag_k(\bfS)\diag(\A))\,, \quad 1\le i\le k\,.\end{equation}
This finishes the proof of \eqref{eq:sldfk1}.

In the case $\alpha\in (2,4)$, we replace $\bfS, S_i$ by $\bfS-c_n\bfI, S_i-c_n$, respectively, and use the same proof as for $\alpha\in (0,2)$.
\end{proof}

Define $\wt L_i, i=1,\ldots,p$ via
\begin{eqnarray*}
(S_{\wt L_i}-c_n) A_{\wt L_i,\wt L_i}&=&\la_i(\diag(\bfS-c_n \bfI)\diag(\A))\,.
\end{eqnarray*}
The random variable $\wt L_i$ encodes the location of the $i$th largest value of the entries of $\diag(\bfS-c_n \bfI)\diag(\bfA)$.

\begin{remark}{\em
As a by-product of the proof of Theorem~\ref{thm:mainsigma2} we get that, with probability tending to 1,
$\{\wt L_1, \ldots,\wt L_k \}= \{L_1,\ldots, L_k\}$ 
for any fixed $k\ge 1$.
}\end{remark}

Next we approximate the eigenvectors of $\Y\Y'$.
To this end, let $\bfe_j=(0,\ldots,0,1,0,\ldots,0)'$, $j=1,\dots,p$, denote the canonical basis vectors of $\R^p$. We define $\sign(\A^{1/2} \bfe_{\wt L_j})$ as the sign of the first non-zero coordinate of the vector $\A^{1/2} \bfe_{\wt L_j}$.

From the point process convergence in Theorem \ref{thm:pp} one can deduce that the largest eigenvalues of $\bfS$ are separated. Indeed they converge in distribution to the $(\Gamma_i^{-2/\alpha})$ in the representation of the limiting point process $N$; see \eqref{eq:N1} and \eqref{eq:N2}. Combining this with Theorem \ref{thm:mainsigma2}, the aforementioned separation property is inherited by the eigenvalues of $\Y\Y'$ which simplifies the identification of associated eigenvectors. It turns out that the unit eigenvectors of $\Y\Y'$ are approximated by the properly normalized $(\A^{1/2} \bfe_{j})$ as shown in the next theorem. 

\begin{theorem}[Eigenvectors of $\Y\Y'$]\label{thm:eigenvector}
Consider the setting and the conditions of Theorem~\ref{thm:mainsigma2}. In addition, we assume $\sigma^2\le M$ a.s.
for some constant $M>0$. 
\begin{enumerate}
\item If $\alpha\in (0,2)$, then
\begin{equation}\label{eq:vec}
\ltwonorm{\bfv_j(\A^{1/2}\bfS \A^{1/2}) - c_{\bfA,j} \A^{1/2} \bfe_{\wt L_j}} \cip 0\,,\qquad \nto\,,\,j\ge 1\,,
\end{equation}
with the normalization and orientation constants
\begin{equation*}
c_{\bfA,j}=\big\|\A^{1/2} \bfe_{\wt L_j}\big\|_{\ell_2}^{-1} \,\, \sign\big(\A^{1/2} \bfe_{\wt L_j}\big)\,.
\end{equation*} 
\item If $\alpha\in (2,4)$, then
\begin{equation*}
\ltwonorm{\bfv_j(\A^{1/2}(\bfS-c_n\bfI) \A^{1/2}) - c_{\bfA,j} \A^{1/2} \bfe_{\wt L_j}} \cip 0\,,\qquad \nto\,,\,j\ge 1\,.
\end{equation*}
\end{enumerate}
\end{theorem}
\begin{proof}
We focus on the case $\alpha \in (0,2)$. 
Recall that $\A^{1/2}\bfS \A^{1/2}$ and $\bfS \A$ have the same eigenvalues. For any eigenvalue $\la$ of $\bfS \A$ with associated eigenvector $\bfv$, i.e $\bfS \A\bfv =\la \bfv$, we have
\begin{equation*}
\A^{1/2}\bfS \A^{1/2} \big(\A^{1/2} \bfv\big) =\la \big(\A^{1/2} \bfv\big)\,.
\end{equation*}
In words, $\bfv$ is an eigenvector of $\bfS \A$ if and only if $\A^{1/2} \bfv$ is an eigenvector of $\A^{1/2}\bfS \A^{1/2}$; and both eigenvectors are associated with the same eigenvalue. For the proof of \eqref{eq:vec}, it is therefore enough to show
\begin{equation}\label{eq:vec1}
\ltwonorm{\bfv_j(\bfS \A) -  \bfe_{\wt L_j}} \cip 0\,,\qquad \nto\,, \,j\ge 1\,.
\end{equation}
Fix $j\ge 1$ and let $(k)$ be the integer sequence from condition \eqref{SB}. We will follow the lines of the proof of Theorem 3.11 in \cite{heiny:mikosch:2017:iid}. 

By Theorem~\ref{thm:diag} and the observation $a_{np}^{-2} \twonorm{\diag(\bfS)-\diag_k(\bfS)} \cip 0$, we see that
\beam \label{eq:colas}
a_{np}^{-2} \max_{i=1,\ldots,p} \ltwonorm{\bfS \A \bfe_i -\diag_k(\bfS) \bfA \bfe_i}
 \le a_{np}^{-2} \twonorm{\bfS \A - \diag_k(\bfS) \bfA} \cip 0\,,\quad\nto\,,
\eeam
and consequently
\begin{equation}\label{eq:mhjf}
\vep^{(n)}:=a_{np}^{-2} \ltwonorm{\bfS \A \, \bfe_{\wt L_j}-S_{\wt L_j} A_{\wt L_j,\wt L_j}\, \bfe_{\wt L_j}} \cip 0\,.
\end{equation}

Before we can apply Proposition~A.7 in \cite{heiny:mikosch:2017:iid} we need to show that, 
with probability converging to $1$, there are no other eigenvalues in a suitably small interval around $\lambda_j(\bfS \A)$.

Let $\xi >1$. We define the set
\begin{equation*}
\Omega_n = \Omega_n(j,\xi)= \{a_{np}^{-2} |\la_j(\bfS \A)-\la_i(\bfS \A)|>\xi\, \vep^{(n)}\, :\, i\neq j =1,\ldots,p  \}\,.
\end{equation*} 
Using \eqref{eq:mhjf} and Theorem~\ref{thm:pp}, we obtain
\begin{equation*}
\begin{split}
\lim_{\nto} \P\big(\Omega_n^c) &= \lim_{\nto} \P( a_{np}^{-2} \min\{ \la_{j-1}(\bfS \A)-\la_{j}(\bfS \A), \la_j(\bfS \A)-\la_{j+1}(\bfS \A)\} \le \xi\, \vep^{(n)}\big) =0\,.
\end{split}
\end{equation*}

From the proof of Theorem~\ref{thm:mainsigma2} recall the definitions of the sets $B_n$ and $C_n$. 
By Proposition~A.7 in \cite{heiny:mikosch:2017:iid}, the unit eigenvector $\bfv_{j}(\bfS\A)$ and the projection $\Proj_{\bfe_{\wt L_j}}(\bfv_{j}(\bfS\A))$
of  the vector $\bfv_{j}(\bfS\A)$ onto the linear space generated by $\bfe_{\wt L_j}$
satisfy for fixed $\delta>0$:
\begin{equation*}
\begin{split}
\limsup_{\nto}~ &\P\big(\ltwonorm{\bfv_{j}(\bfS\A)-\Proj_{\bfe_{\wt L_j}}(\bfv_{j}(\bfS\A))}>\delta\big)\\
 &\le \limsup_{\nto} \P(\{\ltwonorm{\bfv_{j}(\bfS\A)-\Proj_{\bfe_{\wt L_j}}(\bfv_{j}(\bfS\A))}>\delta\} \cap \Omega_n \cap B_n \cap C_n)\\
&\quad + \limsup_{\nto}\P((\Omega_n \cap B_n \cap C_n)^c)\\
&\le \limsup_{\nto} \P(\{2\vep^{(n)}/(\xi\,\vep^{(n)}-\vep^{(n)})>\delta\} \cap \Omega_n \cap B_n \cap C_n)\\
&\le \limsup_{\nto} \P(\{2/(\xi-1)>\delta\})=  \1_{\{ 2/(\xi-1)>\delta\}}.
\end{split}
\end{equation*}
The right-hand side is zero for sufficiently large $\xi$. Since both $\bfv_{j}(\bfS\A)$ and $\bfe_{\wt L_j}$ are unit vectors 
and $\ltwonorm{\Proj_{\bfe_{\wt L_j}}(\bfv_{j}(\bfS\A))}\le 1$, this means that
$\ltwonorm{\bfv_j(\bfS \A) -  \bfe_{\wt L_j}} \cip 0\,.$ This finishes the proof of \eqref{eq:vec1}.

For $\alpha\in (2,4)$, the proof is identical after replacing $\bfS, S_i$ by $\bfS-c_n\bfI, S_i-c_n$, respectively.
\end{proof}


\section{A \sv\ model with thinning}\label{sec:thin}\setcounter{equation}{0}
In this section we consider a modification of the \sv\ model $X_{it}=\sigma_{it}Z_{it}$
introduced in \eqref{eq:svmodel}. We keep the iid structure of the random 
field $(Z_{it})$, the \regvar\ condition \eqref{eq:regvar} on $Z$ and the independence of $(\sigma_{it})$ and $(Z_{it})$ but we 
allow that $\sigma_{it}$ varies with $n$:
\begin{equation}\label{eq:xmult}
X_{it}^{(n)}= \sigma_{it}^{(n)}\, Z_{it}\,,\quad n=1,2,\ldots\,.
\end{equation}
Here $(\sigma_{it}^{(n)})_{i,t\in \N}$ is a field of iid non-negative \rv s  with  a generic element
$\sigma^{(n)}$ whose \ds\ may change with $n$. To be precise, we assume the following condition:

\begin{itemize}
\item[]{\bf Assumption \eqref{eq:sigman}.}  For given $0=s_0<s_1<\cdots < s_m<\infty$ and $m\ge 1$,
\begin{equation}\label{eq:sigman} 
\P(\sigma^{(n)}=s_i)= q_i^{(n)}\,, \quad i=0,\ldots,m, \, n=1,2,\ldots\,, \tag{$A{\sigma}$}
\end{equation}
$\lim_{\nto} q_0^{(n)}=1$ and the limits $\lim_{\nto} n q_i^{(n)}>0$, $i=1,\ldots,m$, exist.
\end{itemize}
\par
\begin{remark}{\em
The restriction to positive $s_i$, $i=1,\ldots,m$, is for notational convenience only.
Also the assumption $\lim_{\nto} q_0^{(n)}=1$ which implies $\sigma^{(n)}\cip 0$ is for simplicity of presentation only. It implies that the matrix $\X$ is sparse. If $\E[(\sn)^{\alpha}]$ had a positive limit $w$, the asymptotic spectral behavior of $\bfS=\X\X'$ constructed from  $\X=(\sigma_{it}^{(n)} Z_{it})$ and $\X=(w^{1/\alpha} Z_{it})$, respectively, would be the same and one could work with the normalizing sequence $a_{np}^2$. However, if $\E[(\sn)^{\alpha}]\to 0$, one needs to take this decay into account and adjust the normalizing sequence to obtain non-trivial asymptotic results.
}\end{remark}
 
We will assume the condition \eqref{eq:p} for some $\beta\in (0,1]$
and use a normalizing \seq\ $(b_n)$ \st
\beao
n\,p\,\E[ (\sigma^{(n)})^\alpha]\,\P(|Z|>b_n)\to 1\,,\qquad \nto\,.
\eeao
Since $q_0^{(n)}\to 1$ we have $\E[(\sigma^{(n)})^{\alpha}]\to 0$.
The additional condition $\lim_{\nto} n\,q_i^{(n)}>0$ means that the expected number of non-zero $\sigma$'s in a row of $\X$ is positive. 
It ensures that $\lim_{\nto}np\,\E[(\sigma^{(n)})^{\alpha}]=\infty$, hence $b_{n}\to\infty$. An alternative way of defining  $(b_n)$ would be
\beam\label{eq:sdgdsd}
b_n = a_{[np\,\E[(\sigma^{(n)})^\alpha]]}\,.
\eeam 
\bre\label{rem:v}
We observe that for any $v>0$,
\beao
\min_{i =1,\ldots,m}s_i^v (1-q_0^{(n)}) \le \E[(\sigma^{(n)})^v]\le \max_{i=1,\ldots,m} s_i^v (1-q_0^{(n)})\,,
\eeao 
hence all moments  $\E[(\sigma^{(n)})^v]$ are of the same order as $1-q_0^{(n)}$. 
\ere
For fixed $n$, relations \eqref{eq:breim} and \eqref{eq:tail2} remain valid but we will need results for these tails when
$x=x_n\to\infty$ as $\nto$. By the uniform \con\ theorem for \regvary\ \fct s we have
(see \eqref{eq:regvar}, \eqref{eq:breim} and \eqref{eq:z1z2} for the definitions of $q_\pm$ and $\wt q_\pm$) 
\beam
\dfrac{\P(\pm \sigma^{(n)}\,Z>x_n)}{\P(|Z|>x_n)}
&\sim&
q_{\pm}\,\E [(\sigma^{(n)})^\alpha]\,,\label{eq:w11}\\
\dfrac{\P(\pm \sigma_1^{(n)}\sigma_2^{(n)}\,Z_1Z_2>x_n)}{\P(|Z_1Z_2|>x_n)}
&\sim& \wt q_{\pm}\,\E [(\sigma_1^{(n)}\sigma_2^{(n)})^{\alpha}]\,.\label{eq:w2}
\eeam
\par
The following result asserts that in the thinned stochastic volatility model \eqref{eq:xmult} the sample covariance matrix is approximated by its diagonal under the new normalization $b_n$. It is an analog of Theorem~\ref{thm:diag}. 
\bth\label{thm:diagsigma} Consider the \sv\ model \eqref{eq:xmult}.
We assume the following conditions:
\begin{itemize} 
\item The \regvar\ condition \eqref{eq:regvar} for some  
$\alpha \in (0,2)\cup (2,4)$ and $\E[Z]=0$ if $\E[|Z|]<\infty$. 
\item The growth condition \eqref{eq:p} for $p=p_n\to\infty$ for some  $\beta \in (0,1] $. 
\item Condition \eqref{eq:sigman} on the \ds\ of $\sigma^{(n)}$.
\end{itemize}
Then
\beam\label{eq:b1}
b_{n}^{-2} \twonorm{\bfS - \diag(\bfS)} \stp 0\,,\qquad\nto\,.
\eeam
\ethe
Theorem~\ref{thm:diag} and Theorem \ref{thm:diagsigma} show that neither the dependence structure in the $\sigma$-field nor a time-dependent distribution of $\sigma$ change the core structure of $\bfS$, which is solely determined by the dependence in the heavy-tailed $Z$-field. Linear dependence among the $Z_{it}$'s, for instance, was studied in \cite{davis:heiny:mikosch:xie:2016}. The resulting approximation of $\bfS$ in this case is block diagonal.
 
The proof of Theorem \ref{thm:diagsigma} is given in Section~\ref{sec:proofthm31}. 
\par
By an application of Weyl's inequality, we may conclude from \eqref{eq:b1} that
\beam\label{eq:b3}
b_{n}^{-2} \max_{i=1,\ldots,p} |\la_{i}(\bfS)-\la_i(\diag(\bfS)|\le b_{n}^{-2}\|\bfS-\diag(\bfS)\|_2\stp 0\,,\qquad \nto\,.
\eeam
\par
Using \eqref{eq:b3} and a continuous mapping argument, we  can derive the limit of the point processes of the
eigenvalues of the sample covariance matrix $\bfS$. 
\bth\label{thm:ppsig}
Assume the conditions of Theorem~\ref{thm:diagsigma} and, 
in addition to \eqref{eq:sigman}, for those $j\in \{1,\ldots,m\}$ for which
 $\lim_{\nto} n\,q_j^{(n)}=\infty$,   
\beam\label{eq:check}
p\,\ex^{-c\,n\,q_j^{(n)}}\to 0\,,\qquad \nto\,,\qquad \mbox{for each $c>0$}.
\eeam
Then we have the following weak \con\ of the \pp es
with state space $\bbr\backslash \{0\}$:
\beao
N_n= \sum_{i=1}^{p} \vep_{b_{n}^{-2}(\la_i(\bfS)-c_{n})}\std N\,,\qquad \nto\,.
\eeao
\item
Here  $N$ is a Poisson process on $\bbr\backslash \{0\}$ with mean \ms\ $\mu_\alpha(x,\infty)=x^{-\alpha/2}$
and $\mu_\alpha(-\infty,-x)=0$ for  $x>0$, and 
\beao
c_n=\left\{ \barr{ll}
0\,,& \mbox{if $\alpha\in (0,2)$,}\\
n\,\E [(X^{(n)})^2]\,,&\mbox{if $\alpha\in (2,4)$}\,. 
\earr\right.
\eeao
\ethe
The proof is given in Section~\ref{sec:proofthm3.2}. This theorem generalizes the results in 
Auffinger and Tang \cite{auffinger:tang:2016} who considered the case $p/n\to \gamma\in (0,\infty)$, $m=1$ and $1-q_0^{(n)}=n^{-v}$ for
some $ v\in [0,1]$. Condition \eqref{eq:check}
ensures that $n\,q_j^{(n)}\to \infty$ sufficiently fast. For example, if $p=n^\beta$ for some $\beta\in (0,1]$ and $q_n^{(j)}\ge n^{-v}$ for some $v\in (0,1)$ then 
for any fixed $c>0$,
\beao
p\, \ex^{-c\, nq_j^{(n)}}\le n^\beta \ex^{-c\,n^{1-v}}\to 0\,.
\eeao  
\par
Theorem~\ref{thm:ppsig} shows that the limiting point processes of the thinned stochastic volatility model and the original one (see Theorem~\ref{thm:pp}) are the same. Typically, thinning decreases the magnitude of the eigenvalues $\la_i(\bfS-c_n)$ which is accounted for by a smaller normalization $b_n$ compared with $a_{np}^2$ used in Theorem~\ref{thm:pp}. Indeed, from \eqref{eq:sdgdsd} one sees that $b_n a_{np}^{-2} \to 0$.
\par

Next, we study the matrix $ \Y=\bfA^{1/2} \X$ and the corresponding sample covariance matrix $\Y\Y'$ under thinning. 

\begin{theorem}\label{thm:mainsigman}
We consider the matrix $\Y=\A^{1/2} \X$, where $\X$ follows the model \eqref{eq:xmult}. We assume the following conditions:
\begin{itemize} 
\item The \regvar\ condition \eqref{eq:regvar} for some  
$\alpha \in (0,2)\cup (2,4)$ and $\E[Z]=0$ if $\E[|Z|]<\infty$. 
\item The growth condition \ref{eq:p} for $p=p_n\to\infty$ for some  $\beta\in (0,1]$. 
\item Condition \eqref{eq:sigman} on the \ds\ of $\sigma^{(n)}$.
\item $\bfA=\A_n$ constitutes a sequence of deterministic, positive definite $p\times p$ matrices with uniformly bounded spectra.
\end{itemize}
 Then
\beao
b_n^{-2}\,\max_{i=1,\ldots,p}\big|\la_{i}(\A^{1/2}\bfS \A^{1/2})-\la_{i}(\diag(\bfS)\A)\big|\stp 0\,.
\eeao
\end{theorem}
The proof of this result is identical to the proof of Theorem~\ref{thm:mainsigma}, using Theorem~\ref{thm:diagsigma} instead of Theorem~\ref{thm:diag}.

Moreover the same arguments that proved Theorems~\ref{thm:mainsigma2} and \ref{thm:eigenvector}, using Theorems \ref{thm:diagsigma} and \ref{thm:mainsigman} instead of Theorems ~\ref{thm:diag} and \ref{thm:mainsigma}, respectively, show the following result.

\begin{theorem}[Eigenvalues and eigenvectors of $\Y\Y'$]\label{thm:sigman}
Consider the setting and the conditions of Theorem~\ref{thm:mainsigman}. In addition, we assume that $(\A_n)$ satisfies condition \eqref{SB}.
\begin{enumerate}
\item If $\alpha\in (0,2)$, we have for the eigenvalues of $\Y\Y'$,
\begin{equation*}
b_n^{-2}\,\max_{i=1,\ldots,p}\big|\la_{i}(\A^{1/2}\bfS \A^{1/2})-\la_i(\diag(\bfS)\diag(\A))\big|\stp 0\,,
\end{equation*}
and for the eigenvectors of $\Y\Y'$,
\begin{equation*}
\ltwonorm{\bfv_j(\A^{1/2}\bfS \A^{1/2}) - c_{\bfA,j} \A^{1/2} \bfe_{\wt L_j}} \cip 0\,,\qquad \nto\,,\,j\ge 1\,,
\end{equation*}
with the normalization and orientation constants
\begin{equation*}
c_{\bfA,j}=\big\|\A^{1/2} \bfe_{\wt L_j}\big\|_{\ell_2}^{-1} \,\, \sign\big(\A^{1/2} \bfe_{\wt L_j}\big)\,.
\end{equation*} 
\item If $\alpha\in (2,4)$, the eigenvalues of $\Y\Y'- \E[\Y\Y']$ satisfy
\begin{equation*}
b_n^{-2}\,\max_{i=1,\ldots,p}\big|\la_{i}(\A^{1/2}(\bfS-c_n \bfI) \A^{1/2})-\la_i(\diag(\bfS-c_n \bfI)\diag(\A))\big|\stp 0\,,
\end{equation*}
and for the eigenvectors of $\Y\Y'- \E[\Y\Y']$ we have
\begin{equation*}
\ltwonorm{\bfv_j(\A^{1/2}(\bfS-c_n\bfI) \A^{1/2}) - c_{\bfA,j} \A^{1/2} \bfe_{\wt L_j}} \cip 0\,,\qquad \nto\,,\,j\ge 1\,.
\end{equation*}
\end{enumerate}
\end{theorem}

In view of Remark~\ref{rem:other}, one can easily extend the results in this section to the case $\beta>1$ in \eqref{eq:p}.

\section{Proof of Theorem~\ref{thm:diag}}\label{sec:proofthm21}\setcounter{equation}{0}
The proof is similar to the one of  Theorem 3.5 in \cite{heiny:mikosch:2017:iid}: one has to replace 
$a_{np}^{-2} Z_{it}$ by $a_{np}^{-2} Z_{it} \sigma_{it}$ and solve a few additional technical difficulties stemming from the dependence in the $\sigma$-field. 
By assumption $\E[Z]=\E[X]=0$ whenever these expections are finite.
Since the Frobenius norm $\frobnorm{\cdot}$ is an upper bound of the spectral norm we have
\begin{equation*}
\begin{split}
a_{np}^{-4}\twonorm{\bfS &- \diag(\bfS)}^2 \le  a_{np}^{-4}\frobnorm{\bfS - \diag(\bfS)}^2\\
&=a_{np}^{-4}\sum_{i, j =1; i\neq j}^p   \sum_{t=1}^n X_{it}^2 X_{jt}^2+a_{np}^{-4}\sum_{i, j =1; i\neq j}^p   \sum_{t_1,t_2=1; t_1\neq t_2}^n X_{i,t_1} X_{j,t_1}X_{i,t_2} X_{j,t_2}\\
&=a_{np}^{-4}\sum_{i, j =1; i\neq j}^p   \sum_{t=1}^n X_{it}^2 X_{jt}^2 \big[\1(|Z_{it} Z_{jt}|  > a_{np}^2)+\1(|Z_{it} Z_{jt}|  \le a_{np}^2)\big]
+I_2^{(n)}\\
&=I_{11}^{(n)}+I_{12}^{(n)}+I_2^{(n)}\,.
\end{split}
\end{equation*}
Thus it suffices to show that each of the expressions on the \rhs\ converges to zero in \pro y.
We have by Markov's inequality for any $\epsilon>0$ and sufficiently small $\delta\in (0,1)$, 
\beao
\P \big( I_{11}^{(n)}>\epsilon\big)&\le&    \sum_{i,j=1,i\ne j}^p\,n\, \P(|Z_{1} Z_{2}| > a_{np}^2)  
\le c\,\dfrac{n\,p^2}{a_{np}^{2\alpha(1-\delta)}}\to 0\,.
\eeao
Here we also used \eqref{eq:z1z2}.
\subsubsection*{The case $\alpha \in (0,2)$.} 
An application of Markov's inequality, finiteness of all moments of $\sigma$ 
and Karamata's theorem for $\alpha<2$
show that for $\vep>0$
\beao
\P\big( I_{12}^{(n)}>\epsilon\big)&\le&
c\, \frac{n}{a_{np}^4} \sum_{i,j=1,i\ne j}^p 
\E[|Z_{1} Z_{2}|^{2}\1(|Z_{1} Z_{2}| \le a_{np}^2)]\\
&\le &c\,n\,p^2 \P(|Z_{1}Z_{2}|>a_{np}^2)\to 0\,,\quad \nto\,.
\eeao

The probability $\P(I_2^{(n)}>\epsilon)$ can be handled in a similar way by applying a 
Karamata argument. 

\subsubsection*{The case $\alpha\in (2,4)$}
Before we proceed we provide an auxiliary result.
Consider the following decomposition 
\beao 
[\bfS - \diag(\bfS)]^2=\bfD+\bfF+\bfR\,,
\eeao
where 
\beao
\bfD=(D_{ij})_{i,j=1,\ldots,p}=\diag([\bfS - \diag(\bfS)]^2)\,.
\eeao
The $p\times p$ matrix $\bfF$ has a zero-diagonal and 
\begin{equation*}
F_{ij}= \sum_{u=1;u\neq i,j}^p \sum_{t=1}^n  X_{it}\, X_{jt}\, X_{ut}^2 ,\quad 1\le i\ne j\le p\,.
\end{equation*}
The $p\times p$ matrix $\bfR$ has a zero-diagonal and 
\begin{equation*}
R_{ij}= \sum_{u=1;u\neq i,j}^p \sum_{t_1=1}^n \sum_{t_2=1; t_2 \neq t_1}^n X_{i,t_1}\, X_{j,t_2}\, X_{u,t_1}\,X_{u,t_2} ,\quad 
1 \le i\neq j\le p\,.
\end{equation*}
The following is the analog of Lemma 4.1 in \cite{heiny:mikosch:2017:iid}. 
\begin{lemma}\label{lem:diagonal} 
Assume the conditions of Theorem~\ref{thm:diag}  and $\alpha\in(2,4)$. 
Then $a_{np}^{-4} \big(\twonorm{\bfD}+\twonorm{\bfF}+\|\bfR\|_2\big)\stp 0$.
\end{lemma}
In view of this lemma we have 
\beao
a_{np}^{-4}\twonorm{\bfS - \diag(\bfS)}^2= 
a_{np}^{-4}\twonorm{[\bfS - \diag(\bfS)]^2}
=a_{np}^{-4}\twonorm{\bfD+\bfF+\bfR}\stp 0\,.
\eeao
This finishes the  proof of Theorem~\ref{thm:diag}. Our final goal is to prove Lemma~\ref{lem:diagonal}.

\begin{proof}[Proof of the $\bfD$-part]
We have for $i=1,\ldots,p$,
\begin{equation*}
\begin{split}
D_{ii}&= \sum_{u=1}^p \sum_{t=1}^n  X_{it}^2 X_{ut}^2 \1(i\neq u)
+\sum_{u=1}^p \sum_{t_1=1}^n \sum_{t_2=1}^n X_{i,t_1} X_{u,t_1}X_{u,t_2} X_{i,t_2} \1(i\neq u)\1(t_1\neq t_2)\\ &= M_{ii}+N_{ii}\,.
\end{split}
\end{equation*}
We write $\bfM$ and $\bfN$ for diagonal matrices constructed from $(M_{ii})$ and $(N_{ii})$ such that $\bfD=\bfM+\bfN$.
First bounding $\|\bfN\|_2$ by the Frobenius norm and then applying Markov's inequality and using the fact that the $Z$'s are centered, one can  prove that $a_{np}^{-4} \twonorm{\bfN}\stp 0$.
Writing $A_{i,u}=\{|\sum_{t=1}^n  X_{it}^2 X_{ut}^2| > a_{np}^2 \}$,
we have for $i=1, \ldots,p$,
\beao
M_{ii}&=& \sum_{u=1,u\ne i}^p \sum_{t=1}^n  X_{it}^2 X_{ut}^2\,\big[\1_{A_{i,u}}+
\1_{A_{i,u}^c}\big]
= M_{ii}^{(1)} + M_{ii}^{(2)}. 
\eeao
On one hand, $\twonorm{\bfM^{(2)}}\le p\, a_{np}^2$. Hence $a_{np}^{-4} \twonorm{\bfM^{(2)}}\stp 0$. On the other hand, 
we obtain with Markov's inequality for $\epsilon>0$ and $r>0$,
\beao
\lefteqn{\P(\twonorm{\bfM^{(1)}} >\epsilon\, a_{np}^{4})}\\ &=& \P(\max_{i=1,\ldots, p} |M_{ii}^{(1)}| >\epsilon\,a_{np}^{4})\\
&\le & \P\Big(\max_{i=1,\ldots, p}\Big|  
\sum_{u=1,u\ne i}^p \sum_{t=1}^n \sigma_{it}^2\sigma_{ut}^2 \1\big(\max_{1\le j\le p,1\le s\le n}\sigma_{js} >(np)^{1/ (4r)}\big)Z_{it}^2 Z_{ut}^2
\,\1_{A_{i,u}}\Big|>\epsilon\, a_{np}^{4}\Big)\\
&&+\P\Big(\max_{i=1,\ldots, p}\Big|  
\sum_{u=1,u\ne i}^p \sum_{t=1}^n \sigma_{it}^2\sigma_{ut}^2 \1\big(\max_{1\le j\le p,1\le s\le n}\sigma_{js} \le (np)^{1/ (4r)}\big)Z_{it}^2 Z_{ut}^2
\,\1_{A_{i,u}}\Big|>\epsilon\, a_{np}^{4}\Big)\\
& \le &np\,\P(|\sigma|>(np)^{1/(4r)})\\&&+
\P\Big(\max_{i=1,\ldots, p}\Big|  
\sum_{u=1,u\ne i}^p \sum_{t=1}^n Z_{it}^2 Z_{ut}^2
\,\1\big(\sum_{t=1}^n Z_{it}^2Z_{ut}^2>a_{np}^4/(np)^{1/r}\big)\Big|>\epsilon\, a_{np}^{4}/(np)^{1/r}\Big)\\
&=& J_1+J_2\,.
\eeao
Since $\E [\sigma^{4r}]<\infty$ we have $J_1\to 0$. We also have for large $n$, sufficiently large $r>0$,
 by the von Bahr and Ess\'een inequality (see Petrov \cite{petrov:1995}, 2.6.20 on p. 82) for $q<\alpha/2$ close to $\alpha/2$, 
\beao
J_2&\le &p^2\,\P\Big(\sum_{t=1}^n Z_{1t}^2Z_{2t}^2>a_{np}^4/(np)^{1/r}\Big) \sim 
p^2\,\P\Big(\sum_{t=1}^n (Z_{1t}^2Z_{2t}^2-(\E [Z^2])^2)>a_{np}^4/(np)^{1/r}\Big) \\
&\le & c\,p^2 \dfrac{(np)^{{q/r}}}{a_{np}^{4q}} \E\Big[\Big|\sum_{t=1}^n (Z_{1t}^2Z_{2t}^2-(\E [Z^2])^2\Big|^q\Big]
\le c\,\dfrac{p^2\, n\,(np)^{{q/r}}}{a_{np}^{4q}}\to 0\,,\qquad \nto\,.
\eeao
\end{proof}

\begin{proof}[Proof of the $\bfF$-~and $\bfR$-parts]
The key observation is that $X=Z\sigma$ is regularly varying with index $\alpha$. Choose $\wt a_n$ such that $\P(X>\wt a_n)\sim n^{-1}$. The sequences $a_n$ and $\wt a_n$ only differ by a slowly varying function which is negligible for the techniques in \cite{heiny:mikosch:2017:iid}. These techniques also work under the dependence stemming from the $\sigma$-field. Therefore the proofs of the $\bfF$-~and $\bfR$-parts are identical to \cite{heiny:mikosch:2017:iid}. 
\end{proof}

\section{Proof of Theorem~\ref{thm:pp}}\label{sec:proofthm24}\setcounter{equation}{0}
In view of \eqref{eq:weyl} a continuous mapping argument shows that the points $(\la_i(\bfS)-c_n)/a_{np}^2$ in $N_n$ may be replaced by
the points $(S_i-c_n)/a_{np}^2$. We denote the resulting \pp\ by
\beao
\wt N_n=\sum_{i=1}^p \vep_{a_{np}^{-2}(S_i-c_n)}\,.
\eeao  
We intend to use Kallenberg's theorem for proving $\wt N_n\std N$; see 
Resnick \cite{resnick:1987}, Proposition 3.22. For this reason, we have to
show the following relations as $\nto$\,,
\beam
\E [\wt N_n(x,\infty)]&\to& \E[N(x,\infty)]= \E[\sigma^\alpha]\,x^{-\alpha/2}\,,\qquad x>0\,,\label{eq:a1}\\
\E [\wt N_n(-\infty,-x)]&\to& \E[N(-\infty,-x)]=0\,,\qquad x>0\,,\label{eq:a2}\\
\P(\wt N_n(e_i,d_i]=0\,,i=1,\ldots,m)&\to& \P(N(e_i,d_i]=0\,,i=1,\ldots,m)\,,\label{eq:a3}
\eeam
where $0<e_1<d_1<\cdots <e_m<d_m<\infty$, $m\ge 1$, are any positive numbers.
We observe that for $S=S_1$,
\beam\label{eq:help1}
\E [\wt N_n(x,\infty)]&=&p\,\P(S>a_{np}^2 x+c_n)\,,\\
\E [\wt N_n(-\infty,-x)]&=&p\,\P(S< -a_{np}^2 x+c_n)\,,\label{eq:help2}
\eeam
Then \eqref{eq:a1} and \eqref{eq:a2}
will be a con\seq\ of the following \ld\ result which is a straightforward application 
of Theorem~4.2 in Mikosch and Wintenberger \cite{mikosch:wintenberger:2013}.
\ble\label{lem:nag}
Assume the conditions of Theorem~\ref{thm:pp}. Write $\gamma_n=n^{2/\alpha+\epsilon}$ for any $\epsilon>0$.
\begin{enumerate}
\item
If $\alpha\in (0,2)$ we have
\beam\label{eq:ld1}
\sup_{y\ge \gamma_n}\Big|\dfrac{\P(S >y)}{n\,\P(X^2>y)}-1\Big|&\to& 0\,.
\eeam
\item
If $\alpha\in (2,4)$ we also assume that
$(\sigma_t)=(\sigma_{it})$ is strongly mixing with rate $(\alpha_j)$  such that $\alpha_j\le c\,j^{-a}$ for some $a>1,c>0$.
Then we have
\beao
\sup_{y\ge \gamma_n}\Big|\dfrac{\P\Big(\sum_{t=1}^n \sigma_{it}^2(Z_{it}^2-\E[Z^2])>y\Big)}{n\,\P(X^2>y)}
-1\Big|&\to& 0\,,\\
\sup_{y\ge \gamma_n}\dfrac{\P\Big(\sum_{t=1}^n \sigma_{it}^2(Z_{it}^2-\E[Z^2])\le -y\Big)}{n\,\P(X^2>y)}
&\to& 0\,.
\eeao
\end{enumerate}
\ele
Then \eqref{eq:a1} and \eqref{eq:a2} follow for $\alpha\in (0,2)$ in view of \eqref{eq:help1}, \eqref{eq:help2}
and by choosing $y=a_{np}^2 x$ in \eqref{eq:ld1}. Indeed, in view of Breiman's lemma,
\beam\label{eq:jj}
p\,\P( S>a_{np}^2x)\sim np\,\P(X^2>a_{np}^2 x)\sim \E[\sigma^{\alpha}] \,np\,\P(|Z|>a_{np} \sqrt{x}) \to\E[\sigma^{\alpha}]\, 
x^{-\alpha/2}\,.
\eeam
The case $\alpha\in (2,4)$ follows in the same way but we also have to show that the \rhs\ in
\beao
p \,\dfrac{\P\Big(a_{np}^{-2}\Big| \sum_{t=1}^n(\sigma_t^2-\E[\sigma^2])\Big|>x\Big)}{n\,\P(X^2>a_{np}^2)}
\sim c\,p^2\,\P\Big(a_{np}^{-2} \Big|\sum_{t=1}^n(\sigma_t^2-\E[\sigma^2])\Big|>x\Big)
\eeao
converges to zero.  By Markov's inequality, the right-hand expression is bounded by 
\beao
c\,x^{-4} \underbrace{\dfrac{(np)^2}{a_{np}^8}}_{\to 0} \E\Big[\Big(n^{-1/2}\sum_{t=1}^n(\sigma_t^2-\E[\sigma^2])\Big)^4\Big]\,.
\eeao
In view of the growth rate of $(\alpha_j)$ and the fact that $\sigma^2\le M$ a.s.,
Theorem 2.5 in \cite{rio:2013} shows that the moments on the \rhs\ converges to a constant, hence
\eqref{eq:jj} converges to zero for $\alpha\in (2,4)$.
\par
Write $\bbf_\sigma$ for the  $\sigma$-algebra generated by $(\sigma_{it})$. 
In what follows, we use the notation $\Ps(\cdot):=\P(\cdot \mid \bbf_\sigma)$ and $\Es[\cdot]:=\E[\cdot \mid \bbf_\sigma]$ for conditional probabilities and expectations with respect to $\bbf_\sigma$.
By independence between $(\sigma_{it})$ and $(Z_{it})$ we have
\beao
\P(\wt N_n(e_i,d_i]=0\,,i=1,\ldots,m)&=& \E\big[\prod_{i=1}^m\Ps(\wt N_n(e_i,d_i]=0)\big]\,.
\eeao
We intend to show that $\wt N_n(e_i,d_i]\std \Pois(\mu_\alpha(e_i,d_i])$ given $\bbf_\sigma$. Then
\eqref{eq:a3} follows. By Poisson's limit theorem (see Billingsley \cite{billingsley:2012}, 
 Theorem 23.2), the latter limit holds if
\beao
\Es \big[\wt N_n(e_i,d_i]\big]\to   \mu_\alpha(e_i,d_i]\,.
\eeao
\ble 
Assume the conditions of Theorem~\ref{thm:pp}. For $\alpha\in (0,2)\cup (2,4)$ and $x>0$, we have
\beam\label{eq:lda}
\Es \big[\wt N_n(x,\infty)\big]&=& \sum_{i=1}^p \Ps \big((S_i-c_n)/a_{np}^2>x \big)\to\mu_\alpha(x,\infty)\,,\\
\Es \big[\wt N_n(-\infty,-x) \big]&=& \sum_{i=1}^p \Ps\big((S_i-c_n)/a_{np}^2<-x\big)\to 0\,.\label{eq:ldb}
\eeam
\ele
\begin{proof} We only show \eqref{eq:lda}, the relation \eqref{eq:ldb} can be proved in a similar way.
We start with the case $\alpha\in (0,2)$ and briefly comment on the case $\alpha\in (2,4)$ at the end of this proof.
We will show that
\beam\label{eq:eqo}
\sup_{i=1,\ldots,p}
\Big|\dfrac{\Ps\big(S_i/a_{np}^2>x\big)}
{\sum_{t=1}^n \sigma_{it}^\alpha\,\P(Z^2>a_{np}^2x)}-1\Big|\to 0\,,\qquad \nto\,.
\eeam
Then by definition of $(a_{np}^2)$ and the ergodic theorem for $(\sigma_{it})$,
\beao
 \sum_{i=1}^p \Ps\big(S_i/a_{np}^2>x\big)&\sim& np\,\P(Z^2>a_{np}^2\,x)
\Big( \dfrac 1 {np}\sum_{i=1}^p\sum_{t=1}^n\sigma_{it}^\alpha \Big)\\
& \to & \E [\sigma^\alpha]\,x^{-\alpha/2}=\mu_\alpha(x,\infty)\,.
\eeao
\par
For ease of presentation,
in the proof of \eqref{eq:eqo} we assume that $x=1$.
Let $i\in \{1,\ldots,p\}$. For small $\epsilon>0$ we have
\beao
\lefteqn{\Ps\big(S_i/a_{np}^2>1 \big)}\\&\le &
\sum_{t=1}^n \Ps\big(\sigma_{it}^2Z_{it}^2> a_{np}^2(1-\epsilon) \big)
+\Ps\big(S_i- \max_{s=1,\ldots,n} \sigma_{is}^2Z_{is}^2>\epsilon a_{np}^2\big)\\
&=&I_{i1}+I_{i2}\,.
\eeao
In view of the uniform \con\ theorem for \regvary\ \fct s and since we assume $\sigma$ to be bounded
we have 
\beam\label{eq:w1}
\lim_{\epsilon\downarrow 0}\limsup_{\nto} \sup_{i=1,\ldots,p}\dfrac{I_{i1}}{\sum_{t=1}^n\sigma_{it}^\alpha\,\P(Z^2>a_{np}^2(1-\epsilon))}\le 1\quad \as
\eeam
For $\delta>0$, we define the counting variable $T_i(\delta)=\sum_{t=1}^n \1(\sigma_{it}^2Z_{it}^2>\delta\,a_{np}^2)$ and consider the disjoint partition 
\begin{equation*}
\{T_i(\delta)\ge 2\}\,, \quad \{T_i(\delta)=1\}\,, \quad \{T_i(\delta)=0\}\,.
\end{equation*}
We have by the same argument as for $I_{i1}$,
\beao
\limsup_{\nto}\sup_{i=1,\ldots,p}\dfrac{\Ps(T_i(\delta)\ge 2)}{\big(\sum_{t=1}^n\sigma_{it}^\alpha \P(Z^2>\delta a_{np}^2)\big)^2}=c(\delta)\quad \as\,,
\eeao
for some constant $c(\delta)$ and therefore the contribution of the set $\{T_i(\delta)\ge 2\}$ is negligible.
Moreover,
\begin{equation*}
\begin{split}
\Ps\Big(T_i(\delta)&=1,S_i- \max_{t=1,\ldots,n} \sigma_{it}^2Z_{it}^2>\epsilon a_{np}^2\Big)
\le \sum_{t=1}^n \Ps(\sigma_{it}^2Z_{it}^2>\delta a_{np}^2, S_i-\sigma_{it}^2Z_{it}^2>\epsilon a_{np}^2)\\
&= \sum_{t=1}^n \Ps(\sigma_{it}^2Z_{it}^2>\delta a_{np}^2) \Ps (S_i-\sigma_{it}^2Z_{it}^2>\epsilon a_{np}^2) = o(1)\,c\,\P(Z^2>a_{np}^2)\,\sum_{t=1}^n\sigma_{it}^\alpha\,,
\end{split}
\end{equation*}
where $o(1)$ does not depend on $i$. Here we used the same argument as for \eqref{eq:w1}.
As regards the set $\{T_i(\delta)=0\}$, we have 
\beao
\lefteqn{\Ps\big(T_i(\delta)=0, S_i- \max_{t=1,\ldots,n} \sigma_{it}^2Z_{it}^2>\epsilon a_{np}^2\big)}\\
&\le &\Ps\big(\max_{t=1,\ldots,n-1}\sigma_{it}^2Z_{it}^2\le \delta
a_{np}^2\,,  S_i- \sigma_{in}^2Z_{in}^2>\epsilon a_{np}^2\big)\\
&\le & \Ps\Big(a_{np}^{-2} \sum_{t=1}^n \sigma_{it}^2 Z_{it}^2 \1(\sigma_{it}^2 Z_{it}^2\le \delta a_{np}^2)>\epsilon \Big)=I_{i3}\,.
\eeao
Since $\sigma^2\le M$ and $p\to\infty$ we have by Karamata's theorem
\beam\label{eq:karam}
a_{np}^{-2}\sum_{t=1}^{n}\Es\big[\sigma_{it}^2 Z^2 \1(\sigma_{it}^2 Z_{it}^2\le \delta a_{np}^2)\big]
&\le &a_{np}^{-2} n\,M\,\E\big[Z^2 \1(M Z^2\le \delta a_{np}^2)\big]\to 0\,. 
\eeam
Hence for large $n$,
\beao
I_{i3}&\le &\Ps\Big(
a_{np}^{-2} \sum_{t=1}^n 
\Big(\sigma_{it}^2 Z_{it}^2  \1(\sigma_{it}^2 Z_{it}^2\le \delta a_{np}^2)
-\Es\big[\sigma_{it}^2 Z^2 \1(\sigma_{it}^2 Z^2\le \delta a_{np}^2)\big]\Big)
>\epsilon/2 \Big)\,.
\eeao
An application of the Fuk-Nagaev inequality (see Petrov \cite{petrov:1995}, p.~78, 2.6.5) yields for $r\ge 2$, $c_1,c_2>0$,
\beao
I_{i3}&\le & a_{np}^{-2r} c_1\sum_{t=1}^n \Es\big[|\sigma_{it} Z|^{2r} \1(\sigma_{it}^2 Z^2\le \delta a_{np}^2)\big]\\
&&+\exp\Big(-c_2 a_{np}^4\Big/\sum_{t=1}^n \var\big(\sigma_{it}^2 Z_{it}^2  \1(\sigma_{it}^2 Z_{it}^2\le \delta a_{np}^2)\mid\bbf_\sigma\big)\Big)\,.
\eeao
An argument similar to \eqref{eq:karam} shows that 
\beao
\limsup_{\nto} \sup_{i=1,\ldots,p}\dfrac{I_{i3}}{ \sum_{t=1}^n\sigma_{it}^\alpha \,\P(Z^2 >a_{np}^2)}=0\quad \as 
\eeao
Summarizing the previous bounds and observing that all of them are uniform in $i$, we proved for given $\epsilon$ 
and sufficiently large $n$ that, with \pro y~1,
\beao
 \sum_{i=1}^p \Ps\big(S_i/a_{np}^2>x\big)\le (1+\epsilon) 
\P(Z^2>a_{np}^{2})\sum_{i=1}^p\sum_{t=1}^n \sigma_{it}^\alpha\,.
\eeao
\par
Next, we show the corresponding lower bound.
In view of the uniform \con\ theorem for \regvary\ \fct s and since we assume $\sigma$ to be bounded
we have for $x=1$ and $\epsilon>0$,
\beao
 \Ps\big(S_i/a_{np}^2>x \big)&\ge &\Ps\big( \max_{t=1,\ldots,n} \sigma_{it}^2Z_{it}^2>(1+\epsilon) a_{np}^2 \big)\\
&\ge &\sum_{t=1}^n 
\Ps\big(\sigma_{it}^2Z_{it}^2> (1+\epsilon) a_{np}^2\big)\\
&&-\sum_{1\le s<t\le n}\Ps\big(\sigma_{it}^2Z^2> (1+\epsilon) a_{np}^2\big)\Ps\big(\sigma_{is}^2Z^2> (1+\epsilon) a_{np}^2\big)\\
&=& \sum_{t=1}^n\sigma_{it}^\alpha \,\,\P(Z^2>a_{np}^2)\, (1+\epsilon)^{-\alpha/2}\, (1+o(1))
\eeao
Since this bound is uniform in $i$, 
we conclude that, for given $\epsilon>0$ and sufficiently large $n$, 
\beao
\sum_{i=1}^p \Ps\big(S_i/a_{np}^2>x\big)\ge (1-\epsilon) 
\P(Z^2>a_{np}^{2})\sum_{i=1}^p\sum_{t=1}^n \sigma_{it}^\alpha\,.
\eeao
This proves the lemma in the case $\alpha\in (0,2)$. 
\par
In the case $\alpha\in (2,4)$, first replace the points $(S_i-c_n)/a_{np}^2$
by $a_{np}^{-2}\sum_{t=1}^n \sigma_{it}^2 (Z_{it}^2-\E [Z^2])$. The argument is similar to the one after Lemma~\ref{lem:nag}. Now one can 
follow the lines of the proof in the case $\alpha\in (0,2)$. We omit details. 
\end{proof}
\section{Proof of Theorem~\ref{thm:diagsigma}}\label{sec:proofthm31}\setcounter{equation}{0}
The proof is similar to the proof of Theorem 3.5 in \cite{heiny:mikosch:2017:iid} and to the proof of Theorem~\ref{thm:diag}.
We will sketch the proof, illustrating the differences one has to pay attention to. We restrict ourselves to the case 
$\alpha\in (0,8/3)\backslash \{2\}$; the case $\alpha\in [8/3,4)$ can be handled in a way  similar to Theorem~\ref{thm:diag}. Indeed, the proof is
even simpler because the field $(\sigma_{it}^{(n)})$ is iid. 
\par
Since the Frobenius norm $\frobnorm{\cdot}$ is an upper bound of the spectral norm we have
\beao
&\phantom{\!\!\!\!}& b_{n}^{-4}\twonorm{\bfS - \diag(\bfS)}^2\le b_{n}^{-4}\frobnorm{\bfS - \diag(\bfS)}^2 \\
 &=& \!\!\!\! b_{n}^{-4}\sum_{i, j =1; i\neq j}^p   \sum_{t=1}^n (X_{it}^{(n)})^2 (X_{jt}^{(n)})^2+b_{n}^{-4}\sum_{i, j =1; i\neq j}^p   
\sum_{t_1,t_2=1; t_1\neq t_2}^n X_{i,t_1}^{(n)} X_{j,t_1}^{(n)}X_{i,t_2}^{(n)} X_{j,t_2}^{(n)}\\
&=& \!\!\!\!b_{n}^{-4} \!\!\!\!\sum_{i, j =1; i\neq j}^p   \sum_{t=1}^n (X_{it}^{(n)})^2 (X_{jt}^{(n)})^2 \big[\1( (X_{it}^{(n)})^2 
(X_{jt}^{(n)})^2  > b_{n}^4)+\1( (X_{it}^{(n)})^2 (X_{jt}^{(n)})^2  \le b_{n}^4 )\big]
+I_2^{(n)}\\
&=& \!\!\!\! I_{11}^{(n)}+I_{12}^{(n)}+I_2^{(n)}\,.
\eeao
Thus it suffices to show that each of the expressions on the \rhs\ converges to zero in \pro y.
By \eqref{eq:w2} and the Potter bounds for \regvary\ \fct s we have for any $\epsilon>0$ and $\nto$, 
\begin{equation*}
\P \big( I_{11}^{(n)}>\epsilon\big)\le   p^2\,n\, \P((X_{1}^{(n)})^2 (X_{2}^{(n)})^2 > b_{n}^4)\sim p^2\,n\, 
(\E [(\sigma^{(n)})^\alpha])^2\,\P(|Z_1Z_2|>b_{n}^2)\to 0\,.
\end{equation*}
Here we also used that $\P(|Z_1Z_2|>x)$ is \regvary\ with index $\alpha$.\\
Assume first $\alpha\in (0,2)$.
Applications of Markov's inequality, Karamata's theorem and the Potter bounds yield
\beao
\lefteqn{\P\big( I_{12}^{(n)}>\epsilon\big)}\\&\le&
c\,\frac{p^2\,n}{b_{n}^4}\, \E[|X_{1}^{(n)} X_{2}^{(n)}|^{2}\1(|X_{1}^{(n)} X_{2}^{(n)}| \le b_{n}^2)]\\
&=&c\, p^2 \,n\,\sum_{i,j=1}^m q_i^{(n)}q_j^{(n)}\,
\dfrac{s_i^2s_j^2 \E[(Z_1Z_2)^2\1(s_is_j |Z_1Z_2|\le b_{n}^2)]}{b_{n}^4\,
\P(s_is_j |Z_1Z_2|> b_{n}^2)}\, \P(s_is_j |Z_1Z_2|> b_{n}^2)\\
&\sim &c\, p^2 \,n\,\sum_{i,j=1}^m q_i^{(n)}q_j^{(n)} s_i^{\alpha}\,s_j^{\alpha}\P(|Z_1Z_2|> b_{n}^2)\,\\
&=&c\, p^2 \,n \,(\E [(\sigma^{(n)})^{\alpha}])^2\,\P(|Z_1Z_2|> b_{n}^2)\to 0\,,\quad \nto\,.
\eeao
If $\alpha\in (2,8/3)$ we have $\E[Z^2]<\infty$. Hence
\beao
\P\big( I_{12}^{(n)}>\epsilon\big)&\le&
c\,\frac{p^2\,n}{b_{n}^4}\, \E[|(X_{1}^{(n)} X_{2}^{(n)}|^{2}]
= c\,\frac{p^2\,n}{b_{n}^4}(\E[(\sigma^{(n)})^2])^2\to 0\,.
\eeao
Here we also used the fact that all moments of $\sigma ^{(n)}$ are of the same size; see Remark~\ref{rem:v}.
\par
For $\alpha\in (0,2)$, the probability $P_{2}^{(n)}=\P(I_2^{(n)}>\epsilon)$ can be handled analogously; we omit details.
We turn to $P_{2}^{(n)}$ in the case $\alpha \in (2,8/3)$. In particular, we have $\E[Z]=0$ and $\E [Z^2]<\infty$. 
With \v Cebychev's inequality, also using the independence and the fact that $\E [X^{(n)}]=0$,  we find that
\beao
P_{2}^{(n)} &\le& c \frac{1}{b_{n}^8} \E\Big[ \Big( \sum_{i, j =1; i\neq j}^p   \sum_{t_1,t_2=1; t_1\neq t_2}^n X_{i,t_1}^{(n)} X_{j,t_1}^{(n)}X_{i,t_2}^{(n)} X_{j,t_2}^{(n)}\Big)^2 \Big]\\&\le& 
c\, \frac{(p\,n)^2}{b_{n}^8 }\,\big(\E[(\sigma^{(n)})^2]\big)^4\to 0\,,\quad \nto\,.
\eeao
This finishes the proof.
\section{Proof of Theorem~\ref{thm:ppsig}}\label{sec:proofthm3.2}\setcounter{equation}{0}
In what follows, we will write $S$ for a generic element
of the \seq\ of diagonal entries $(S_i)$. 
Since we have
\beao
b_{n}^{-2}\max_{i=1,\ldots,p}\|(\la_{i}(\bfS)-c_n)-(\la_i(\diag(\bfS))-c_n)\|_2\stp 0\,,\qquad\nto\,,
\eeao
a continuous mapping argument shows that it suffices to show the \pp\ \con\
\beao
\wt N_n= \sum_{i=1}^p \vep_{b_{n}^{-2} (S_i-c_n)}\std N\,,\qquad \nto\,.
\eeao
Since the points $(S_i)$ are independent it suffices to show that for $x>0$,
\beam\label{eq:b5}
\E[\wt N_n(x,\infty)]&=&p\,\P(S>x b_{n}^2 +c_n)\to \E[N(x,\infty)]=x^{-\alpha/2}\,,\\
\E[\wt N_n(-\infty,x)]&=&p\,\P(S<-x b_{n}^2+c_n)\to \E[N(-\infty,-x)]=0\,.\label{eq:b6}
\eeam
We restrict ourselves to prove \eqref{eq:b5}; the proof of \eqref{eq:b6} is analogous.
For generic \seq s $(Z_t)$ and $(\sigma_t^{(n)})$ we have the \rep\
\beao
S= \sum_{j=1}^m s_j^2\sum_{t=1}^n Z_t^2 \1(\sigma_t^{(n)}=s_j)=\sum_{j=1}^ms_j^2 \sum_{t\in A_j} Z_t^2\,,
\eeao
where $A_j=A_j^{(n)}= \{1\le t\le n: \sigma_t^{(n)}=j\}$. Write $M_j$ for the cardinality of $A_j$. Then we have the \rep
\beao
S\eqd \sum_{j=1}^m s_j^2\,T_j\,,\quad\mbox{where
$T_j=\sum_{t=1}^{M_j} Z_{jt}^2$\,,}
\eeao
and $(M_j)$ and $(Z_{jt})_{t=1,2,\ldots;j=1,\ldots,m}$ are independent. We observe that $M_j$ is binomially distributed 
with mean $\E[M_j]= n\,q_j^{(n)}$. The next lemma concludes the proof of Theorem~\ref{thm:ppsig}.
\ble Assume the conditions of Theorem~\ref{thm:ppsig}. Then \eqref{eq:b5} holds.
\ele
\begin{proof} 
Define $\tau_j= \lim_{\nto} n\,q_j^{(n)}$, $j=1,\ldots,m$. 
We will consider two cases:
\begin{enumerate}
\item
At least one $\tau_j$ is infinite.
\item
All $\tau_j$ are finite.
\end{enumerate}
Throughout we assume $\alpha\in (0,2)$; the case $\alpha\in (2,4)$ is analogous, taking into account
the centering $c_n$ for $S$.  
\par
We start with the case that $\tau_k=\infty$. If $0<\tau_j<\infty$ for some 
$j\ne k$ we will show that $s_j^2 T_j$ does not contribute to  $\lim_{\nto}p\P\big(S>x\,b_n^2\big)$. In this case, $M_j\std Y_j\sim \Pois(\tau_j)$ 
and  $\E\big[\ex^{hM_j}\big] \to \E[\ex^{hY_j}]$, $h>0$.
We have by Markov's inequality for positive $h,\epsilon$,
\beao
p\,\P(s_j^2 T_j> \epsilon b_n^2)&=&p\,\P(Z^2>b_n^2)\,\sum_{k=1}^\infty \P(M_j=k)\,
\dfrac{\P\Big(\sum_{t=1}^k Z_t^2 >\epsilon b_n^2\Big)}{\P(Z^2>b_n^2)}\\
&\le &p\,\P(Z^2>b_n^2)\,\E\big[\ex^{hM_j}\big]\sum_{k=1}^\infty \ex^{-hk}
\dfrac{\P\Big(\sum_{t=1}^k Z_t^2 >\epsilon b_n^2\Big)}{\P(Z^2>b_n^2)}\\
&\sim &  \epsilon^{-\alpha/2} \dfrac 1 {n\E[(\sigma^{(n)})^\alpha]}\underbrace{pn\,\E[(\sigma^{(n)})^\alpha] \P(Z^2>b_n^2)}_{\sim 1} \,\E\big[\ex^{hY_j}\big]\sum_{k=1}^\infty k\,\ex^{-hk}\to 0\,.
\eeao
Here we also used the subexponential property of the \ds\ of $Z^2$ (see Theorem A3.20 in Embrechts et al. \cite{embrechts:kluppelberg:mikosch:1997}).
\par
Therefore we assume for the rest of the proof of case (1) 
that $\tau_j=\infty$ for all $1\le j\le m$.
\par
We have for small $\epsilon>0$, 
\beam\label{eq:upper}
\P\big(S>x\,b_n^2\big)\le \sum_{j=1}^m\P\big(s_j^2 T_j>xb_n^2(1-\epsilon)\big)
+\P\Big( \bigcap_{k=1}^m |S- s_k^2 T_k|>\epsilon x b_n^2\Big)=I_1+I_2\,. 
\eeam
First we deal with $I_1$.
We notice that  $b_n^2/  (nq_j^{(n)})^{2/\alpha}\to\infty$. Our goal is to apply classical large deviation results 
(see Theorem A.1 in \cite{heiny:mikosch:2017:iid}) after replacing $M_j$ by $\E[M_j]$. We have for small~$\delta$,
\beao
J_j&=&\P\big(s_j^2 T_j>xb_n^2(1-\epsilon)\big)\\&= &\P\big(s_j^2 T_j>xb_n^2(1-\epsilon)\,,|M_j-\E[M_j]|\le \delta \E[M_j]\big)\\
&&+ \P\big(s_j^2 T_j>xb_n^2(1-\epsilon)\,,|M_j-\E[M_j]| >\delta \E[M_j]\big)=J_{j1}+J_{j2}\,.
\eeao
We have
\beao
J_{j2}\le \P\big(|M_j-\E[M_j]| >\delta \E[M_j]\big)=\P\big(M_j>(1+\delta) \E[M_j]\big)+\P\big(M_j<(1-\delta)\, \E[M_j])\big)\,.
\eeao
An application of Markov's exponential inequality yields for $h=\log(1+\delta)$
and small $\delta>0$, 
\beao
\P\big(M_j>(1+\delta) \E[M_j]\big)&\le& \ex^{-h(1+\delta)\,nq_j^{(n)}}
\Big(1-q_j^{(n)} (1-\ex^h) \Big)^n\\&\le& 
\ex^{-nq_j^{(n)}\big(h(1+\delta)+(1-\ex^h)\big)}\\
&=&\ex^{-nq_j^{(n)}\big( (1+\delta)\log (1+\delta)-\delta)\big)}\\
&\le & \ex^{-0.5 \,\delta^2 nq_j^{(n)}}\,.
\eeao 
A similar argument shows that for small $\delta>0$,
\beao
\P\big(M_j<(1-\delta)\, \E[M_j])\big)\le  \ex^{-0.5 \,\delta^2 nq_j^{(n)}}
\eeao
In view of condition \eqref{eq:check} we have
\beao
p\, J_{j2}\le 2\,p\,\ex^{-0.5 \,\delta^2 nq_j^{(n)}}\to 0\,,\qquad \nto\,.
\eeao

We also have in view of  Theorem A.1 in \cite{heiny:mikosch:2017:iid}
\beao
J_{j1}&\le &\P\Big(s_j^2\sum_{t=1}^{(1+\delta) \E[M_j]}Z_t^2>xb_n^2(1-\epsilon)\Big)\sim (1+\delta) \E[M_j]
\P(s_j^2 Z^2 >x b_n^2(1-\epsilon))\\&\sim& x^{-\alpha/2}\dfrac{1+\delta}{(1-\epsilon)^{\alpha/2}} n\,s_j^\alpha q_j^{(n)}\P(|Z|>b_n)\,,\\
J_{j1}&\ge &\P\Big(s_j^2\sum_{t=1}^{(1-\delta) \E[M_j]}Z_t^2>xb_n^2(1-\epsilon)\Big)
\sim  x^{-\alpha/2}\dfrac{1-\delta}{(1-\epsilon)^{\alpha/2}} n\,s_j^\alpha q_j^{(n)}\P(|Z|>b_n)\,.
\eeao
Letting $\delta\downarrow 0$ and recalling the definition of $b_n$, we conclude that
\beao
\lim_{\epsilon\to 0}\limsup_{\nto} p\,I_1=\lim_{\epsilon\to 0}\limsup_{\nto} p\,\sum_{j=1}^m J_j= x^{-\alpha/2}\,.
\eeao
Our next goal is to show that $pI_2\to 0$.
Consider a disjoint partition for small $\delta>0$ and $j=1,\ldots,m$,
\beao
B_{1}&=&\bigcup_{1\le i<j\le m}\big\{s_i^2T_i>\delta b_{n}^2\,,s_j^2T_j>\delta b_{n}^2\big\}\,,\\
B_{2}&=&\bigcup_{j=1}^m \big\{s_j^2T_j>\delta b_{n}^2\,, s_i^2T_i\le\delta b_{n}^2\,,
i\ne j\,,i=1,\ldots,m\big\}\,,\\
B_{3}&=&\big\{ \max_{j\le m} s_j^2T_j\le \delta b_{n}^2\big\}\,.
\eeao
We have 
\begin{equation*}
p\,\P(B_1) \le p \sum_{1\le i<j\le m} \P(s_i^2T_i>\delta b_{n}^2\,,s_j^2T_j>\delta b_{n}^2)\,.
\end{equation*}
To show that the \rhs~converges to $0$, we proceed as for $J_j.$ 
For $i<j$ we replace  the random indices $M_i$ and $M_j$ in $T_i$ and $T_j$ by their 
corresponding expectations. We omit further details. Abusing notation here and
in what follows, we denote the resulting 
modified quantities by the same symbols $T_i$ and $T_j$.
After this operation, $T_i$ and $T_j$ are independent and we can treat their tail \pro ies
in the same way as for $J_j$, yielding $\lim_{\nto}p\P(B_1)=0$.
\par
Next we observe that
\beao
\P(\{ |S-s_j^2T_j|>\epsilon b_n^2\,,j\le m\}\cap B_2)
&\le &\sum_{j=1}^m\P( |S-s_j^2T_j|>\epsilon b_n^2\,,s_j^2T_j>\delta b_n^2)\,.
\eeao
Now proceed as for $J_j$: replace all $M_j$ by $\E[M_j]$ in each \pro y in the sum.  Then the modified sums $S-s_j^2T_j$ and $s_j^2T_j$ become
independent. Using the independence, we see that
\beao\lefteqn{
\limsup_{\nto} p\,\P(\{ |S-s_j^2T_j|>\epsilon b_n^2\,,j\le m\}\cap B_2)}\\
&=&\limsup_{\nto} \dfrac 1p\sum_{j=1}^m \big(p\,\P( |S-s_j^2T_j|>\epsilon b_n^2)\big)\,\big(p\,\P(s_j^2T_j>\delta b_n^2)\big)
=0\,.
\eeao

Finally, we deal with 
\beao
\P(\{ |S-s_j^2T_j|>\epsilon b_n^2\,,j\le m\}\cap B_3)&\le& 
\P\Big(b_n^{-2}\sum_{j=1}^m s_j^2T_j \1(s_j^2T_j\le \delta b_n^2) >\epsilon\Big)\\
&\le &\sum_{j=1}^m \P\Big(b_n^{-2} s_j^2T_j \1(s_j^2T_j\le \delta b_n^2) >\epsilon/m\Big)\,.
\eeao
Since we can choose $\delta$ independently from $\epsilon$, we can take $\delta <\epsilon/m$, making the \rhs\ vanish.
Combining all the previous bounds, we finally arrived at 
\beao
\limsup_{\nto} p\,\P(S>x b_n^2)\le x^{-\alpha/2}\,,\qquad x>0\,,
\eeao
in the case $\alpha\in (0,2)$. In the case $\alpha\in (2,4)$ we have to center
the quantities $S$ and $T_j$. Then the same ideas of the proof apply, in particular the large deviations results
of Theorem A.1 in \cite{heiny:mikosch:2017:iid}. We omit details.
\par
Next consider, for $\alpha\in (0,2)$,
\beam\label{eq:lowerbound}
\P\big(S>x\,b_n^2\big)&\ge& \P\big(s_j^2 T_j>xb_n^2(1+\epsilon)\,,
|S-s_j^2T_j|\le \epsilon b_n^2 \,\mbox{for some $j\le m$}\big)\nonumber\\
&\ge &\sum_{j=1}^m \P\big(s_j^2 T_j>xb_n^2(1+\epsilon)\,,
|S-s_j^2T_j|\le \epsilon b_n^2 \big)\nonumber\\
&&-\sum_{1\le i<j\le m} \P\big(s_i^2 T_i>xb_n^2(1+\epsilon)\,,
s_j^2 T_j>xb_n^2(1+\epsilon)\big)\,.
\eeam
We proceed as before: we replace the numbers $M_j$ by their expecations. After this operation the modified sums
$s_j^2T_j$, $S-s_j^2T_j$  and $s_i^2T_i$ for $i\ne j$ become independent. Moreover, $\P(|S-s_j^2T_j|\le \epsilon b_n^2)\to 1$. Hence for fixed small $\delta>0$ and large $n$,
\beao
\P\big(S>x\,b_n^2\big)\ge (1-\delta)\sum_{j=1}^m\P\big(s_j^2 T_j>xb_n^2(1+\epsilon)\big)\,.
\eeao
Applying  Theorem A.1 in \cite{heiny:mikosch:2017:iid} and letting $\epsilon,\delta$ go to zero, 
we proved that
\beao
\liminf_{\nto} p\,\P\big(S>x\,b_n^2\big)\ge x^{-\alpha/2}\,, \quad x>0\,.
\eeao
\par
Our next goal is to consider case (2) in which $0<\tau_j<\infty$  for all $j$.
We will show that
\beam
p\,\P( S>xb_n^2) &\sim& p\,\sum_{j=1}^m \P(s_j^2 T_j>xb_n^2)\nonumber\\&\sim& 
p\,\sum_{j=1}^ms_j^\alpha\,\E[M_j] \P(Z^2>xb_n^2)\label{eq:t1}\\
&=&pn \E[(\sigma^{(n)})^\alpha] \P(Z^2>xb_n^2)\to x^{-\alpha/2}\,.\nonumber
\eeam
We have $M_j\std Y_j\sim \Pois(\tau_j)$ as $\nto$, in particular $\P(M_j=k)\to \pi_k^{(j)}=\P(Y_j=k)$
and $\P(M_j=k)\le c\,\ex^{-hk}$, $k\ge 1$, $h>0$; see \cite{embrechts:kluppelberg:mikosch:1997} p.~41, equation (1.31).
Keeping this in mind, subexponentiality of the \ds\ of $Z^2$  yields 
\beam\label{eq:green}
\dfrac{\P(s_j^2T_j>x b_n^2)}{\P(Z^2>b_n^2)}&=& \sum_{k=1}^\infty \P(M_j=k)\,\dfrac{\P\big(s_j^2\sum_{t=1}^kZ_t^2>xb_n^2\big)}
{\P(Z^2>b_n^2)}\nonumber\\
&\to &  s_j^\alpha\sum_{k=1}^\infty \pi_k^{(j)}\,k =s_j^\alpha\E[M_j]\,,\qquad \nto\,.
\eeam
For the upper bound in \eqref{eq:t1} 
we recall the inequality \eqref{eq:upper}. In view of \eqref{eq:green} and \regvar\ of $Z^2$, for the upper bound it 
remains to show that
\beam\label{eq:ss1}
\dfrac{\P(s_j^2T_j>b_n^2\,,j=1,\ldots,m)}{\P(Z^2>b_n^2)}&\to& 0\,.
\eeam
We show \eqref{eq:ss1} only for $m=2$. We have
\begin{equation*}
\begin{split}
\P(s_1^2T_1 &> b_n^2,s_2^2T_2>b_n^2)=\sum_{k,l=1}^\infty \P(M_1=k,M_2=l)\, \P\big(s_1^2\sum_{t=1}^k Z_t^2>b_n^2\big)\P\big(s_2^2\sum_{t=1}^l Z_t^2>b_n^2\big)\\
&\le \sum_{k=1}^\infty \sqrt{\P(M_1=k)} \P\big(\sum_{t=1}^k Z_t^2>b_n^2\big)\sum_{l=1}^\infty \sqrt{\P(M_2=l)}\P\big(\sum_{t=1}^l Z_t^2>b_n^2\big)\,.
\end{split}
\end{equation*}
The same arguments which established \eqref{eq:green} show that the \rhs\ is of the order $O((\P(Z^2>b_n^2))^2)$. This proves \eqref{eq:ss1}.
\par
The lower bound in \eqref{eq:t1} follows by similar arguments, taking into
account the inequality \eqref{eq:lowerbound}.

\end{proof}

\section*{Acknowledgments}\setcounter{equation}{0}

We thank Richard Davis, Olivier Wintenberger and Mark Podolskij for inspiring discussions.

\bibliographystyle{acm}
{ \bibliography{libraryjohannes}}
\end{document}